\documentclass[11pt]{amsart} 
\usepackage[linkcolor=blue, citecolor = blue]{hyperref}
\hypersetup{colorlinks=true}

\usepackage{amsthm,amsfonts,amsmath,amssymb,amscd} 

\usepackage{verbatim}
\usepackage{float}
\usepackage{wrapfig}
\usepackage{mathtools}
\usepackage[color=orange!20, linecolor=orange]{todonotes}

\usepackage{caption, subcaption}
\usepackage{enumitem}

\usepackage{ytableau}

\usepackage{tikz}
\usetikzlibrary{decorations.markings}
\tikzstyle{vertex}=[circle, draw, inner sep=0pt, minimum size=4pt]

\usetikzlibrary{arrows,automata}
\usepackage{tikz, tikz-3dplot, pgfplots}
\usetikzlibrary[positioning,patterns]
\usepackage{tikz-3dplot}
\usepackage{wasysym}

\newtheorem{theorem}{Theorem} 

\newtheorem{lemma}[theorem]{Lemma}
\newtheorem{corollary}[theorem]{Corollary}
 
\theoremstyle{definition}
\newtheorem{definition}[theorem]{Definition}
\newtheorem{example}[theorem]{Example}

\theoremstyle{remark}
\newtheorem{remark}{Remark}

\newcommand{\sgn}{\mathrm{sgn}}

\newcommand{\drawpoints}[7]{ 
    \foreach \px in { #1,...,#2 } {
        \foreach \py in { #3,...,#4 } {
            \foreach \pz in { #5,...,#6 } {
                \filldraw[#7] (\px, \py, \pz) circle (0.7pt);
            };
        };
    };
}

\newcommand{\drawrectZ}[7][cycle]{ 
    \draw [#2] (#4, #6, #3) -- (#4, #7, #3) -- (#5, #7, #3) -- (#5, #6, #3) -- #1;
}

\newcommand{\drawplaneZ}[5]{ 
    \foreach \i in { 0,..., #4 } {
        \draw [-, #1] (#2, \i, #5) -- (#3, \i, #5);
    }
    \foreach \i in { #2,..., #3 } {
        \draw [-, #1] (\i, 0, #5) -- (\i, #4, #5);
    }
}

\newcommand{\drawplaneX}[5]{ 
    \foreach \i in { 0,..., #5 } {
        \draw [#1] (#2, 0, \i) -- (#2, 4, \i);
    }
    \foreach \i in { #3,..., #4 } {
        \draw [#1] (#2, \i, 0) -- (#2, \i, #5);
    }
}

\newcommand{\showXZweak}[3]{ 
    \foreach \pz in {#2, ..., -1} {
        \draw[->, #3] (#1,0,\pz) -- (#1, 0, \pz + 1);
        \draw[->, #3] (#1,0,\pz) -- (#1 - 1, 0, \pz + 1);
    }
}
\newcommand{\showAllXZ}[4]{ 
    \foreach \px in {#1, ..., #2}{
        \foreach \pz in {#3, ..., -1} {
            \draw[->, #4] (\px,0,\pz) -- (\px, 0, \pz + 1);
            \draw[->, #4] (\px,0,\pz) -- (\px - 1, 0, \pz + 1);
        }
    }
}

\newcommand{\picturenpath}{
\begin{tikzpicture}[x=1cm,y=1cm,z=0.5cm,rotate around y=40,rotate around x=0, rotate around z=0, scale = 0.7]\label{graph1}

    \draw[dashed, ->] (xyz cs:x=0) -- (xyz cs:x=5.5) node[above] {$x$};
    \draw[dashed, ->, very thin, lightgray] (xyz cs:y=0) -- (xyz cs:y=4.3) node[right] {$y$};
    \draw[->, dashed] (xyz cs:z=0) -- (xyz cs:z=-5) node[above] {$z$};
    \drawplaneZ{lightgray}{-1}{1}{4}{-2}
    \drawplaneZ{lightgray}{0}{3}{4}{-1}
    \drawplaneZ{lightgray}{1}{4}{4}{0}
    \showAllXZ{0}{4}{-4}{ultra thin, dashed, lightgray}

    \node[above] at (1,4,-2){${B_1}$};
    \node[above] at (2,4,-1){${B_2}$};
    \node[above] at (3,4,-1){${B_3}$};
    \node[above] at (4,4,0){${B_4}$};
    \node[below] at (0,0,-4){${A_1}$};
    \node[below] at (1,0,-3){${A_2}$};
    \node[below] at (2,0,-3){${A_3}$};
    \node[below] at (3,0,-2){${A_4}$};
    \draw [blue, very thick] (0, 0, -4) -- (-1, 0, -3);
    \draw [blue, very thick] (-1, 0, -3) -- (-1, 0, -2) node[above, midway]{$t_3$};
    \draw [black, very thick] (-1, 0, -2) -- (-1, 1, -2);
    \draw [black, very thick] (-1, 1, -2) -- (-1, 1, -2);
    \draw [black, very thick] (-1, 1, -2) -- (0, 2, -2) node[left, pos=0.7]{$x_3$};
    \draw [black, very thick] (0, 2, -2) -- (0, 3, -2);
    \draw [->,black, very thick] (0, 3, -2) -- (1, 4, -2);
    \draw [blue, very thick] (1, 0, -3) -- (0, 0, -2);
    \draw [blue, very thick] (0, 0, -2) -- (0, 0, -1) node[above, midway]{$t_2$};
    \draw [black, very thick] (0, 0, -1) -- (0, 0, -1);   
    \draw [black, very thick] (0, 0, -1) -- (0, 1, -1);
    \draw [black, very thick] (0, 1, -1) -- (0, 2, -1);
    \draw [black, very thick] (0, 2, -1) -- (1, 3, -1) node[left, pos=0.7]{$x_2$};
    \draw [->,black, very thick] (1, 3, -1) -- (2, 4, -1) node[left, pos=0.7]{$x_1$};

    \draw [blue, very thick] (2, 0, -3) -- (2, 0, -2) node[above, midway]{$t_3$};
    \draw [blue, very thick] (2, 0, -2) -- (1, 0, -1);
    \draw [black, very thick] (1, 0, -1) -- (1, 1, -1);
    \draw [black, very thick] (1, 1, -1) -- (2, 2, -1) node[left, pos=0.7]{$x_3$};
    \draw [black, very thick] (2, 2, -1) -- (3, 3, -1) node[left, pos=0.7]{$x_2$};
    \draw [->,black, very thick] (3, 3, -1) -- (3, 4, -1);

    \draw [blue, very thick] (3, 0, -2) -- (2, 0, -1);
    \draw [blue, very thick] (2, 0, -1) -- (1, 0, 0);
    \draw [black, very thick] (1, 0, 0) -- (2, 1, 0) node[left, pos=0.7]{$x_4$};
    \draw [black, very thick] (2, 1, 0) -- (3, 2, 0) node[left, pos=0.7]{$x_3$};
    \draw [black, very thick] (3, 2, 0) -- (4, 3, 0) node[left, pos=0.7]{$x_2$};
    \draw [->,black, very thick] (4, 3, 0) -- (4, 4, 0);

    \end{tikzpicture}
}

\newcommand{\picturenpathy}{
\begin{tikzpicture}[x=1cm,y=1cm,z=0.5cm,rotate around x=-63,rotate around z=-24, scale = 0.4]\label{graph2}    
    \node [black] at (-0.5, 4.8, -0.5) {\large $\mu$};
    \draw[-, lightgray, dashed, thick, fill] (0, 4, 0) --
                            (0, 4, -2) --
                            (1, 4, -2) --
                            (1, 4, -1) --
                            (3, 4, -1) --
                            (3, 4, 0) --
                            cycle;

    \node [black] at (4, 0, -3.5) {\large $\lambda$};
    \draw[-, black, thick, dashed] (0, 0, -5) --
                            (1, 0, -5) --
                            (1, 0, -4) --
                            (3, 0, -4) --
                            (3, 0, -3) --
                            (4, 0, -3) --
                            (4, 0, 0);
    \draw[dashed, ->] (xyz cs:x=0) -- (xyz cs:x=6) node[above] {$x$};
    \draw[->, dashed] (xyz cs:z=0) -- (xyz cs:z=-6) node[right] {$z$};

    \filldraw[blue] (0,0,-4) circle (1pt) node[left, black] {${A_1}$};
    \filldraw[blue] (1,0,-3) circle (1pt) node[below, black] {${A_2}$};
    \filldraw[blue] (2,0,-3) circle (1pt) node[below, black] {${A_3}$};
    \filldraw[blue] (3,0,-2) circle (1pt) node[below, black] {${A_4}$};
    \filldraw[black] (1,4,-2) circle (1pt) node[above] {${B_1}$};
    \filldraw[black] (2,4,-1) circle (1pt) node[above] {${B_2}$};
    \filldraw[black] (3,4,-1) circle (1pt) node[above] {${B_3}$};
    \filldraw[black] (4,4,0) circle (1pt) node[above] {${B_4}$};
    
    \draw [blue, very thick] (0, 0, -4) -- (-1, 0, -3);
    \draw [blue, very thick] (-1, 0, -3) -- (-1, 0, -2);
    \draw [black, very thick] (-1, 0, -2) -- (-1, 1, -2);
    \draw [black, very thick] (-1, 1, -2) -- (-1, 1, -2);
    \draw [black, very thick] (-1, 1, -2) -- (0, 2, -2);
    \draw [black, very thick] (0, 2, -2) -- (0, 3, -2);
    \draw [->,black, very thick] (0, 3, -2) -- (1, 4, -2);

    \draw [blue, very thick] (1, 0, -3) -- (0, 0, -2);
    \draw [blue, very thick] (0, 0, -2) -- (0, 0, -1);
    \draw [black, very thick] (0, 0, -1) -- (0, 0, -1);   
    \draw [black, very thick] (0, 0, -1) -- (0, 1, -1);
    \draw [black, very thick] (0, 1, -1) -- (0, 2, -1);
    \draw [black, very thick] (0, 2, -1) -- (1, 3, -1);
    \draw [->,black, very thick] (1, 3, -1) -- (2, 4, -1);

    \draw [blue, very thick] (2, 0, -3) -- (2, 0, -2);
    \draw [blue, very thick] (2, 0, -2) -- (1, 0, -1);
    \draw [black, very thick] (1, 0, -1) -- (1, 1, -1);
    \draw [black, very thick] (1, 1, -1) -- (2, 2, -1);
    \draw [->, black, very thick] (2, 2, -1) -- (3, 3, -1);

    \draw [blue, very thick] (3, 0, -2) -- (2, 0, -1);
    \draw [blue, very thick] (2, 0, -1) -- (1, 0, 0);
    \draw [black, very thick] (1, 0, 0) -- (2, 1, 0);
    \draw [black, very thick] (2, 1, 0) -- (3, 2, 0);
    \draw [->, black, very thick] (3, 2, 0) -- (4, 3, 0);

    \end{tikzpicture}
}

\newcommand{\picturegraph}{
\begin{tikzpicture}[x=1cm,y=1cm,z=0.5cm,rotate around y=-45,rotate around x=-0, scale = 0.8]\label{graph2}

    \showXZweak{3}{-3}{very thin, lightgray, dashed}
    \showXZweak{2}{-3}{very thin, lightgray, dashed}
    \showXZweak{1}{-3}{very thin, lightgray, dashed}
   
    \drawplaneZ{dashed, lightgray, thin}{0}{3}{3}{0};
    \draw [->, black, very thick] (2,2,0) -- (2, 3, 0);
    \draw [->, black, very thick] (2,2,0) -- (3, 3, 0);
    
    \drawplaneZ{gray, very thin}{0}{3}{3}{-2};
    \draw [->, black, very thick] (2,2,-2) -- (2, 3, -2);
    \draw [->, black, very thick] (2,2,-2) -- (3, 3, -2);

    \draw[->, dashed] (xyz cs:x=0) -- (xyz cs:x=4) node[above] {$x$};
    \draw[->, dashed] (xyz cs:y=0) -- (xyz cs:y=4) node[right] {$y$};
    \draw[->, dashed] (xyz cs:z=0) -- (xyz cs:z=-3.5) node[above] {$z$};

    \draw [->, blue, very thick] (3,0,-3) -- (3, 0, -2);
    \draw [->, blue, very thick] (3,0,-3) -- (2, 0, -2);

    \end{tikzpicture}
}

\newcommand{\drawprojections}[1]{
    \draw [#1] (1, 0, -1) -- (2, 1, -1) node[black, below, sloped, midway]{{$z = 1$}};
    \draw [#1] (2, 1, -1) -- (2, 2, -1);
    \draw [#1] (2, 2, -1) --  (3, 3, -1);
    \draw [#1] (3, 3, -1) --  (4, 4, -1);

    \draw [#1] (2, 0, -2) -- (3, 1, -2) node[black, below, sloped, midway]{{$z = 2$}};
    \draw [#1] (3, 1, -2) -- (3, 2, -2);
    \draw [#1] (3, 2, -2) --  (4, 3, -2);
    \draw [#1] (4, 3, -2) --  (4, 4, -2);
    
    \draw [#1] (3, 0, -3) -- (4, 1, -3) node[black, below, sloped, midway]{{$z = 3$}};
    \draw [#1] (4, 1, -3) -- (4, 2, -3);
    \draw [#1] (4, 2, -3) --  (4, 3, -3);
    \draw [#1] (4, 3, -3) --  (4, 4, -3);

    \draw [#1] (3, 0, -4) -- (4, 1, -4) node[black, below, sloped, midway]{{$z = 4$}};
    \draw [#1] (4, 1, -4) -- (4, 2, -4);
    \draw [#1] (4, 2, -4) --  (4, 3, -4);
    \draw [#1] (4, 3, -4) --  (4, 4, -4);

    \draw [#1] (4, 0, -5) -- (4, 1, -5) node[black, below, sloped, midway]{{$z = 5$}};
    \draw [#1] (4, 1, -5) -- (4, 2, -5);
    \draw [#1] (4, 2, -5) --  (4, 3, -5);
    \draw [#1] (4, 3, -5) --  (4, 4, -5);
}

\newcommand{\pictureproj}{
\begin{tikzpicture}[->,x=1cm,y=1.2cm,z=1cm,rotate around y=-65,rotate around x=0, scale=0.6]
    
    \draw[->, dashed] (xyz cs:y=0) -- (xyz cs:y=5) node[right] {$y$};
    \draw[->, dashed] (xyz cs:z=0) -- (xyz cs:z=-6) node[above] {$z$};
    
    \filldraw[black] (3, 0, -4) circle (2pt) node[below] {$A_j$};
    \filldraw[black] (4, 4, 0) circle (2pt) node[above] {$B_i$};
    \filldraw[black] (4, 0, -5) circle (2pt) node[below] {$F_i$};
    \filldraw[black] (4, 4, -1) circle (2pt) node[right] {$T_{i,1}$};
    \draw [->, gray, dashed, very thin] (1,0,-1) -- (4, 0, -1);
    \draw [->, gray, dashed, very thin] (2,0,-2) -- (4, 0, -2);
    \draw [->, gray, dashed, very thin] (3,0,-3) -- (4, 0, -3);
    \draw [->, gray, dashed, very thin] (3,0,-4) -- (4, 0, -4);

    \drawplaneZ{dashed, gray, very thin}{0}{4}{4}{0};
    \drawplaneX{-, red, dashed, very thin}{4}{0}{4}{-5};
    \filldraw[black] (4, 4, -2) circle (1pt);
    \filldraw[black] (4, 4, -3) circle (1pt);
    \filldraw[black] (4, 4, -4) circle (1pt);
    \filldraw[black] (4, 4, -5) circle (1pt);

    \draw [-, blue, very thick] (3, 0, -4) --  (3, 0, -3) node[pos=0.8, below]{$t_4$};
    \draw [-, blue, very thick] (3, 0, -3) --  (2, 0, -2);
    \draw [-, blue, very thick] (2, 0, -2) --  (1, 0, -1);
    \draw [-, blue, very thick] (1, 0, -1) --  (1, 0, 0) node[pos=0.8, below]{$t_1$};
    \draw [-, gray, very thick] (1, 0, 0) -- (2, 1, 0) node[black, below, sloped, midway]{{z = 0}} node[black, left, pos=0.6]{$x_4$};
    \draw [-, gray, very thick] (2, 1, 0) --  (2, 2, 0);
    \draw [-, gray, very thick] (2, 2, 0) --  (3, 3, 0) node[black, left, pos=0.6]{$x_2$};
    \draw [->, gray, very thick] (3, 3, 0) --  (4, 4, 0) node[black, left, pos=0.6]{$x_1$};

    \drawprojections{-, black, thick};
    
    \draw [blue, dashed, thick] (4, 0, -5) -- (3, 0, -4);
    
    \draw [-, red, very thick]  (4, 4, 0) -- (4, 4, -1); 
    \draw [-, red, very thick] (4, 4, -1) -- (4, 3, -2); 
    \draw [-, red, very thick] (4, 3, -2) -- (4, 2, -3); 
    \draw [-, red, very thick] (4, 1, -3) -- (4, 1, -4); 
    \draw [-, red, very thick] (4, 1, -4) -- (4, 0, -5); 

    \draw [<-, dashed, lightgray, very thick] (4,1,-3) -- (4.5, 1, -4) node[right, black]{$T_{i,3}$};
    \draw [<-, dashed, lightgray, very thick] (4,2,-3) -- (4.5, 2, -4) node[right, black]{$T'_{i,3}$};
    \draw [<-, dashed, lightgray, very thick] (4,3,-2) -- (4.5, 3, -4) node[right, black]{$T_{i,2}$};

    \draw [black, dashed, thin] (3,3,0)
                                -- (3,3,-1)
                                -- (4,3,-2);
    \draw [black, dashed, thin] (2,2,0)
                                -- (2,2,-1)
                                -- (3,2,-2)
                                -- (4,2,-3);
    \draw [black, dashed, thin] (2,1,0)
                                -- (2,1,-1)
                                -- (3,1,-2)
                                -- (4,1,-3);

\end{tikzpicture}
}

\newcommand{\drawprojectionss}[1]{
    
        \draw [#1] (2, 0, -2) -- (3, 1, -2) node[black, below, sloped, midway]{{$z = k$}};
        \draw [#1] (3, 1, -2) -- (3, 2, -2);
        \draw [#1] (3, 2, -2) --  (4, 3, -2);
        \draw [#1] (4, 3, -2) --  (4, 4, -2);
        
    }

\newcommand{\pictureprojj}{
\begin{tikzpicture}[->,x=1cm,y=1.2cm,z=1cm,rotate around y=-65,rotate around x=0, scale = 0.6]
    
    \draw[->, dashed] (xyz cs:y=0) -- (xyz cs:y=5) node[right] {$y$};
    \draw[->, dashed] (xyz cs:z=1) -- (xyz cs:z=-6) node[above] {$z$};
    
    \filldraw[black] (3, 0, -4) circle (2pt) node[below] {$A_j$};
    \filldraw[black] (4, 4, 0) circle (2pt) node[above] {$B_i$};
    \filldraw[black] (4, 3, -2) circle (2pt) node[right] {$T_{i,k}$};

    \drawplaneZ{dashed, lightgray, very thin}{0}{4}{4}{0};
    \drawplaneX{-, red, dashed, very thin}{4}{0}{4}{-5};
    \filldraw[black] (4, 4, -2) circle (1pt);

    \drawplaneZ{gray, very thin}{0}{4}{4}{-2};

    \draw [-, blue, very thick] (3, 0, -4) --  (3, 0, -3);
    \draw [-, blue, very thick] (3, 0, -3) --  (2, 0, -2);
    \filldraw[black] (2, 0, -2) circle (2pt) node[below] {$C$};
    \draw [-, blue, very thick] (2, 0, -2) --  (1, 0, -1);
    \draw [-, blue, very thick] (1, 0, -1) --  (1, 0, 0);
    \filldraw[black] (1, 0, 0) circle (2pt) node[below] {};
    \draw [-, gray, very thick] (1, 0, 0) -- (2, 1, 0) node[black, below, sloped, midway]{{$z = \mu'_i$}};
    \draw [-, gray, very thick] (2, 1, 0) --  (2, 2, 0);
    \draw [-, gray, very thick] (2, 2, 0) --  (3, 3, 0);
    \draw [->, gray, very thick] (3, 3, 0) --  (4, 4, 0);
    
    \draw[<-,dashed, gray, thick] (2,0,-2) -- (1,0,0);
    \draw[<-,dashed, gray, thick] (4,3,-2) -- (3,3,0);
   \node[sloped] at (2.9,6.5,-7.5) {$x = i$};
    
    \drawprojectionss{-,black, very thick};

\end{tikzpicture}
}

\newcommand{\picturestepk}{
\begin{tikzpicture}[x=1cm,y=1cm,z=0.7cm,rotate around y=0,rotate around x=0,rotate around z = 0, scale=1]\label{picturestepk}
    \draw[->, dashed, gray, very thin] (xyz cs:x=0) -- (xyz cs:x=6) node[above] {$x$};
    \draw[->, dashed, gray, very thin] (xyz cs:y=0) -- (xyz cs:y=7) node[right] {$y$};
    \draw[->, dashed, gray, very thin] (xyz cs:z=0) -- (xyz cs:z=-4.5) node[above] {$z$};
    \drawpoints{0}{6}{0}{0}{-1}{-4}{black}
    \drawplaneZ{lightgray, very thick}{0}{6}{7}{-1}

    \node [fill=white] at (5,6,-1) {$z = k$};

    \draw [-, very thick, blue] (1, 0, -4) -- (1, 0, -3);
    \draw [-, very thick, blue] (1, 0, -3) -- (1, 0, -2);
    \draw [->, very thick, blue] (1, 0, -2) -- (0, 0, -1) node(e1)[pos=0.4, left]{$e_{i}$};
    \draw [-, very thick, blue] (0, 0, -1) -- (0, 1, -1);
    \draw [-, very thick, blue] (0, 1, -1) -- (0, 2, -1);
    \draw [-, very thick, blue] (0, 2, -1) -- (1, 3, -1);
    \draw [->, very thick, blue] (1, 3, -1) -- (2, 4, -1) node(e11)[pos=0.4, above]{$e'_{i}$};
    \draw [-, very thick, blue] (2, 4, -1) -- (2, 5, -1);
    \draw [-, very thick, blue] (2, 5, -1) -- (3, 6, -1);
    \draw [-, very thick, blue] (3, 6, -1) -- (3, 7, -1);
    \draw [-, very thick, black] (4, 0, -4) -- (3, 0, -3);
    \draw [-, very thick, black] (3, 0, -3) -- (2, 0, -2);
    \draw [->, very thick, black] (2, 0, -2) -- (2, 0, -1) node(e2)[pos=0.4, right]{$e_{i+1}$};
    \draw [-, very thick, black] (2, 0, -1) -- (2, 1, -1);
    \draw [-, very thick, black] (2, 1, -1) -- (3, 2, -1);
    \draw [-, very thick, black] (3, 2, -1) -- (3, 3, -1);
    \draw [->, very thick, black] (3, 3, -1) -- (3, 4, -1) node(e12)[pos=0.6, right]{$e'_{i+1}$};
    \draw [-, very thick, black] (3, 4, -1) -- (4, 5, -1);
    \draw [-, very thick, black] (4, 5, -1) -- (4, 6, -1);
    \draw [->, very thick, black] (4, 6, -1) -- (4, 7, -1);
    \draw [-, very thick, blue, dashed] (1, 0, -4) -- (1, 0, -3);
    \draw [-, very thick, blue, dashed] (1, 0, -3) -- (1, 0, -2);
    \draw [-, very thick, blue, dashed] (1, 0, -2) -- (1, 0, -1);
    \draw [-, very thick, blue, dashed] (1, 0, -1) -- (1, 1, -1);
    \draw [-, very thick, blue, dashed] (1, 1, -1) -- (1, 2, -1);
    \draw [-, very thick, blue, dashed] (1, 2, -1) -- (2, 3, -1);
    \draw [-, very thick, blue, dashed] (2, 3, -1) -- (2, 4, -1);
    \draw [-, very thick, blue, dashed] (2, 4, -1) -- (2, 5, -1);
    \draw [-, very thick, blue, dashed] (2, 5, -1) -- (3, 6, -1);
    \draw [->, very thick, blue, dashed] (3, 6, -1) -- (3, 7, -1);
    \draw [-, very thick, black, dashed] (4, 0, -4) -- (3, 0, -3);
    \draw [-, very thick, black, dashed] (3, 0, -3) -- (2, 0, -2);
    \draw [-, very thick, black, dashed] (2, 0, -2) -- (1, 0, -1);
    \draw [-, very thick, black, dashed] (1, 0, -1) -- (1, 1, -1);
    \draw [-, very thick, black, dashed] (1, 1, -1) -- (2, 2, -1);
    \draw [-, very thick, black, dashed] (2, 2, -1) -- (2, 3, -1);
    \draw [-, very thick, black, dashed] (2, 3, -1) -- (3, 4, -1);
    \draw [-, very thick, black, dashed] (3, 4, -1) -- (4, 5, -1);
    \draw [-, very thick, black, dashed] (4, 5, -1) -- (4, 6, -1);
    \draw [-, very thick, black, dashed] (4, 6, -1) -- (4, 7, -1);
    
    \filldraw[blue] (1, 0, -4) circle (1pt) node(A1)[anchor=south east]{$A_{j}$};
    \filldraw[black] (4, 0, -4) circle (1pt) node(A2)[anchor=south west]{$A_{s}$};
    \filldraw[blue] (3, 7, -1) circle (1pt) node(B1)[anchor=south]{$B_{i}$};
    \filldraw[black] (4, 7, -1) circle (1pt) node(B2)[anchor=south]{$B_{i+1}$};
    \filldraw[blue] (2, 4, -1) circle (1pt) node(C1)[anchor=south east]{$C_1$};
    \filldraw[blue] (1, 3, -1) circle (1pt) node(D1)[anchor=south east]{$D_1$};
    \filldraw[black] (3, 4, -1) circle (1pt) node(C2)[anchor=south east]{$C_2$};
    \filldraw[black] (3, 3, -1) circle (1pt) node(D2)[anchor=north west]{$D_2$};
    \filldraw[black] (2, 3, -1) circle (1pt) node(D)[anchor=north west]{$D$};
\end{tikzpicture}
}

\newcommand{\picturestepkinv}{
\begin{tikzpicture}[x=1cm,y=1cm,z=0.7cm,rotate around y=0,rotate around x=0, rotate around z = 0]\label{picturestepkinv}
    \draw[->, dashed, gray, very thin] (xyz cs:x=0) -- (xyz cs:x=6) node[above] {$x$};
    \draw[->, dashed, gray, very thin] (xyz cs:y=0) -- (xyz cs:y=7) node[right] {$y$};
    \draw[->, dashed, gray, very thin] (xyz cs:z=0) -- (xyz cs:z=-4.5) node[above] {$z$};
    \drawpoints{0}{6}{0}{0}{-1}{-4}{black}
    \drawplaneZ{lightgray, very thick}{0}{6}{7}{-2}

    \node [fill=white] at (5,6,-2) {$z = k + 1$};

    \draw [-, very thick, blue] (1, 0, -4) -- (1, 0, -3);
    \draw [-, very thick, blue] (1, 0, -3) -- (1, 0, -2);
    \draw [-, very thick, blue] (1, 0, -2) -- (1, 1, -2);
    \draw [-, very thick, blue] (1, 1, -2) -- (1, 2, -2);
    \draw [-, very thick, blue] (1, 2, -2) -- (2, 3, -2);

    \draw [-, very thick, blue] (3, 4, -2) -- (3, 5, -2);
    \draw [-, very thick, blue] (3, 5, -2) -- (3, 6, -2);
    \draw [->, very thick, blue] (3, 6, -2) -- (3, 7, -2);
    
    \draw [-, very thick, black] (4, 0, -4) -- (3, 0, -3);
    \draw [-, very thick, black] (3, 0, -3) -- (2, 0, -2);
    \draw [-, very thick, black] (2, 0, -2) -- (2, 0, -2);
    \draw [-, very thick, black] (2, 0, -2) -- (2, 1, -2);
    \draw [-, very thick, black] (2, 1, -2) -- (3, 2, -2);
    \draw [-, very thick, black] (3, 2, -2) -- (3, 3, -2);
    
    \draw [-, very thick, black] (3, 4, -2) -- (4, 5, -2);
    \draw [-, very thick, black] (4, 5, -2) -- (4, 6, -2);
    \draw [->, very thick, black] (4, 6, -2) -- (4, 7, -2);
    
    \draw [-, very thick, blue, dashed] (1, 0, -4) -- (1, 0, -3);
    \draw [-, very thick, blue, dashed] (1, 0, -3) -- (1, 0, -2);
    \draw [-, very thick, blue, dashed] (1, 0, -2) -- (1, 1, -2);
    \draw [-, very thick, blue, dashed] (1, 1, -2) -- (1, 2, -2);
    \draw [-, very thick, blue, dashed] (1, 2, -2) -- (2, 3, -2);
    \draw [-, very thick, blue, dashed] (2, 3, -2) -- (2, 4, -2);
    \draw [-, very thick, blue, dashed] (2, 4, -2) -- (2, 5, -2);
    \draw [-, very thick, blue, dashed] (2, 5, -2) -- (3, 6, -2);
    \draw [-, very thick, blue, dashed] (3, 6, -2) -- (3, 7, -2);
    
    \draw [-, very thick, black, dashed] (4, 0, -4) -- (3, 0, -3);
    \draw [-, very thick, black, dashed] (3, 0, -3) -- (2, 0, -2);
    \draw [-, very thick, black, dashed] (2, 0, -2) -- (2, 1, -2);
    \draw [-, very thick, black, dashed] (2, 1, -2) -- (3, 2, -2);
    \draw [-, very thick, black, dashed] (3, 2, -2) -- (3, 3, -2);
    \draw [-, very thick, black, dashed] (3, 3, -2) -- (4, 4, -2);
    \draw [-, very thick, black, dashed] (4, 4, -2) -- (4, 5, -2);
    \draw [-, very thick, black, dashed] (4, 5, -2) -- (4, 6, -2);
    \draw [-, very thick, black, dashed] (4, 6, -2) -- (4, 7, -2);
    
    \draw [->, very thick, blue] (2, 3, -2) -- (3, 4, -2) node(e11)[pos=0.4, above]{$e'_{i}$};
    \draw [->, very thick, black] (3, 3, -2) -- (3, 4, -2) node(e12)[inner sep=1.75pt, pos=0.65, anchor=west, fill=white]{$e'_{i+1}$};
    \draw [->, very thick, blue] (1, 0, -2) -- (0, 0, -1) node(e1)[pos=0.25, left]{$e_{i}$};
    \draw [->, very thick, black] (2, 0, -2) -- (2, 0, -1) node(e2)[pos=0.25, right]{$e_{i+1}$};
    \draw [-, very thick, blue, dashed] (1, 0, -2) -- (1, 0, -1);
    \draw [-, very thick, black, dashed] (2, 0, -2) -- (1, 0, -1);
    
    \filldraw[black] (3, 4, -2) circle (1pt) node(C)[anchor=south east]{$C$};
    \filldraw[blue] (2, 3, -2) circle (1pt) node(D1)[anchor=south east]{$D'_1$};
    \filldraw[black] (3, 3, -2) circle (1pt) node(D2)[anchor=north west]{$D'_2$};
    \filldraw[blue] (1, 0, -4) circle (1pt) node(A1)[anchor=south east]{$A_{j}$};
    \filldraw[black] (4, 0, -4) circle (1pt) node(A2)[anchor=south west]{$A_{s}$};
    \filldraw[blue] (3, 7, -1) circle (1pt) node(B1)[anchor=south]{$B_{i}$};
    \filldraw[black] (4, 7, -1) circle (1pt) node(B2)[anchor=south]{$B_{i+1}$};
    \draw [->, very thick, blue, dashed] (3, 7, -1) -- (3, 7, -2);
    \draw [->, very thick, black, dashed] (4, 7, -1) -- (4, 7, -2);

\end{tikzpicture}
}

\newcommand{\picturetransBefore}{
\begin{tikzpicture}[x=1cm,y=0.7cm,z=0.7cm,rotate around y=35,rotate around x=0, rotate around z = 0]\label{picturetransafter}
    \draw[->, dashed, lightgray, very thin] (xyz cs:x=0) -- (xyz cs:x=7) node[black, above] {$x$};
     \draw[->, dashed, lightgray, very thin] (xyz cs:y=0) -- (xyz cs:y=8) node[black, right] {$y$};
    \draw[->, dashed, lightgray, very thin] (xyz cs:z=0) -- (xyz cs:z=-5) node[black, above] {$z$};
    \drawpoints{0}{6}{0}{0}{0}{-4}{black}
    \drawplaneZ{lightgray, thin}{1}{4}{7}{-2}
    \drawplaneZ{lightgray, thin}{-1}{4}{7}{0}

    \node [fill=white] at (0,6,-2) {$z = k$};
    \node [fill=white] at (0,6,0) {$z = k'$};
    \draw [-, very thick, blue] (1, 0, -4) -- (1, 0, -3);
    \draw [-, very thick, blue] (1, 0, -3) -- (1, 0, -2);
    \draw [-, very thick, blue] (1, 0, -2) -- (1, 1, -2);
    \draw [-, very thick, blue] (1, 1, -2) -- (1, 2, -2);
    \draw [-, very thick, blue] (1, 2, -2) -- (2, 3, -2);
    \draw [-, very thick, blue] (2, 3, -2) -- (3, 4, -2);
    \draw [-, very thick, blue] (3, 4, -2) -- (3, 5, -2);
    \draw [-, very thick, blue] (3, 5, -2) -- (3, 6, -2);
    \draw [->, very thick, blue] (3, 6, -2) -- (3, 7, -2);
    
    \draw [-, semithick, blue, dashed] (1, 0, -2) -- (0, 0, -1);
    \draw [-, semithick, blue, dashed] (0, 0, -1) -- (-1, 0, 0);
    \draw [-, semithick, blue, dashed] (-1, 0, 0) -- (-1, 1, 0);
    \draw [-, semithick, blue, dashed] (-1, 1, 0) -- (-1, 2, 0);
    \draw [-, semithick, blue, dashed] (-1, 2, 0) -- (0, 3, 0);
    \draw [->, semithick, blue, dashed] (0, 3, 0) -- (1, 4, 0);

    \draw [-, very thick, black] (4, 0, -4) -- (3, 0, -3);
    \draw [-, very thick, black] (3, 0, -3) -- (2, 0, -2);
    \draw [-, very thick, black] (2, 0, -2) -- (2, 0, -2);
    \draw [-, very thick, black] (2, 0, -2) -- (1, 0, -1);
    
    \draw [-, very thick, black] (1, 0, -1) -- (0, 0, 0);
    \draw [-, very thick, black] (0, 0, 0) -- (0, 1, 0);
    \draw [-, very thick, black] (0, 1, 0) -- (0, 2, 0);
    \draw [-, very thick, black] (0, 2, 0) -- (1, 3, 0);
    \draw [-, very thick, black] (1, 3, 0) -- (1, 4, 0);
    \draw [->, very thick, black] (1, 4, 0) -- (2, 5, 0);
    \draw [-, very thick, red] (2, 5, 0) -- (3, 6, 0);
    \draw [->, very thick, red] (3, 6, 0) -- (4, 7, 0);

    \draw [-, thick, black, dashed] (2, 0, -2) -- (2, 1, -2);
    \draw [-, thick, black, dashed] (2, 1, -2) -- (2, 2, -2);
    \draw [-, thick, black, dashed] (2, 2, -2) -- (3, 3, -2);
    \draw [-, thick, black, dashed] (3, 3, -2) -- (3, 4, -2);
    \draw [-, thick, black, dashed] (3, 4, -2) -- (4, 5, -2);
    \draw [-, thick, black, dashed] (4, 5, -2) -- (4, 6, -2);
    \draw [->, thick, black, dashed] (4, 6, -2) -- (4, 7, -2);
    
    \filldraw[black] (3, 4, -2) circle (1pt) node[anchor=south east](C){$C$};
    \filldraw[black] (1, 4, 0) circle (1pt) node[anchor=south east](C){$C'$};
    \filldraw[black] (1, 0, -2) circle (1pt) node(E1)[anchor=south east]{$E_i$};
    \filldraw[black] (2, 0, -2) circle (1pt) node(E2)[right]{$E_j$};
    \filldraw[black] (0, 0, 0) circle (1pt) node(E2)[right]{$F_j$};
    
    \filldraw[black] (1, 0, -4) circle (1pt) node(A1)[below]{$A_{\ell}$};
    \filldraw[black] (4, 0, -4) circle (1pt) node(A2)[below]{$A_{r}$};
    \filldraw[black] (3, 7, -2) circle (1pt) node[anchor=south] (B1){$B_{i}$};
    \filldraw[black] (4, 7, 0) circle (1pt) node[anchor=south](B2){$B_{j}$};
\end{tikzpicture}
}

\newcommand{\picturetransAfter}{
\begin{tikzpicture}[x=1cm,y=0.7cm,z=0.7cm,rotate around y=35,rotate around x=0, rotate around z = 0]\label{picturetransafter}
    \draw[->, dashed, lightgray, very thin] (xyz cs:x=0) -- (xyz cs:x=7) node[black, above] {$x$};
    \draw[->, dashed, lightgray, very thin] (xyz cs:y=0) -- (xyz cs:y=8) node[black, right] {$y$};
    \draw[->, dashed, lightgray, very thin] (xyz cs:z=0) -- (xyz cs:z=-5) node[black, above] {$z$};
    \drawpoints{0}{6}{0}{0}{0}{-4}{black}
    \drawplaneZ{lightgray, thin}{1}{4}{7}{-2}
    \drawplaneZ{lightgray, thin}{-1}{4}{7}{0}

    \node [fill=white] at (0,6,-2) {$z = k$};
    \node [fill=white] at (0,6,0) {$z = k'$};
    \draw [-, very thick, blue] (1, 0, -4) -- (1, 0, -3);
    \draw [-, very thick, blue] (1, 0, -3) -- (1, 0, -2);
    
    \draw [-, thick, blue, dashed] (1, 0, -2) -- (1, 1, -2);
    \draw [-, thick, blue, dashed] (1, 1, -2) -- (1, 2, -2);
    \draw [-, thick, blue, dashed] (1, 2, -2) -- (2, 3, -2);
    \draw [-, thick, blue, dashed] (2, 3, -2) -- (3, 4, -2);
    \draw [-, very thick, black] (3, 4, -2) -- (3, 5, -2);
    \draw [-, very thick, black] (3, 5, -2) -- (3, 6, -2);
    \draw [->, very thick, black] (3, 6, -2) -- (3, 7, -2);
    
    \draw [-, very thick, blue] (1, 0, -2) -- (0, 0, -1);
    \draw [-, very thick, blue] (0, 0, -1) -- (-1, 0, 0);
    \draw [-, very thick, blue] (-1, 0, 0) -- (-1, 1, 0);
    \draw [-, very thick, blue] (-1, 1, 0) -- (-1, 2, 0);
    \draw [-, very thick, blue] (-1, 2, 0) -- (0, 3, 0);
    \draw [->, very thick, blue] (0, 3, 0) -- (1, 4, 0);

    \draw [-, very thick, black] (4, 0, -4) -- (3, 0, -3);
    \draw [-, very thick, black] (3, 0, -3) -- (2, 0, -2);
    \draw [-, very thick, black] (2, 0, -2) -- (2, 0, -2);

    \draw [-, thick, black, dashed] (2, 0, -2) -- (1, 0, -1);
    \draw [-, thick, black, dashed] (1, 0, -1) -- (0, 0, 0);
    \draw [-, thick, black, dashed] (0, 0, 0) -- (0, 1, 0);
    \draw [-, thick, black, dashed] (0, 1, 0) -- (0, 2, 0);
    \draw [-, thick, black, dashed] (0, 2, 0) -- (1, 3, 0);
    \draw [-, thick, black, dashed] (1, 3, 0) -- (1, 4, 0);
    \draw [->, very thick, blue] (1, 4, 0) -- (2, 5, 0);
    \draw [-, very thick, red] (2, 5, 0) -- (3, 6, 0);
    \draw [->, very thick, red] (3, 6, 0) -- (4, 7, 0);
  
    \draw [-, very thick, black] (2, 0, -2) -- (2, 1, -2);
    \draw [-, very thick, black] (2, 1, -2) -- (2, 2, -2);
    \draw [-, very thick, black] (2, 2, -2) -- (3, 3, -2);
    \draw [->, very thick, black] (3, 3, -2) -- (3, 4, -2);
    
    \filldraw[black] (3, 4, -2) circle (1pt) node[anchor=south east](C){$C$};
    \filldraw[black] (1, 4, 0) circle (1pt) node[anchor=south east](C){$C'$};
    \filldraw[black] (1, 0, -2) circle (1pt) node(E1)[anchor=south east]{$E_i$};
    \filldraw[black] (2, 0, -2) circle (1pt) node(E2)[right]{$E_j$};
    \filldraw[black] (0, 0, 0) circle (1pt) node(E2)[right]{$F_j$};
    
    \filldraw[black] (1, 0, -4) circle (1pt) node(A1)[below]{$A_{\ell}$};
    \filldraw[black] (4, 0, -4) circle (1pt) node(A2)[below]{$A_{r}$};
    \filldraw[black] (3, 7, -2) circle (1pt) node[anchor=south] (B1){$B_{i}$};
    \filldraw[black] (4, 7, 0) circle (1pt) node[anchor=south](B2){$B_{j}$};
\end{tikzpicture}
}

\newcommand{\pictureslideI}{
\begin{tikzpicture}[x=1.2cm,y=1cm,z=0.5cm,rotate around y=33,rotate around x=0]\label{slide1}

    \draw[->, dashed, lightgray] (xyz cs:x=-3) -- (xyz cs:x=5.3) node[above] {$x$};
    \draw[->, dashed, lightgray] (xyz cs:y=0) -- (xyz cs:y=4.3) node[right] {$y$};
    \draw[->, dashed, lightgray] (xyz cs:z=0) -- (xyz cs:z=-4.3) node[above] {$z$};
    \showAllXZ{0}{4}{-4}{ultra thin, dashed, lightgray};
    \drawrectZ{gray, thin}{-1}{-1}{5}{0}{4}
    \drawrectZ{gray, thin}{-2}{-1}{5}{0}{4}
    \drawrectZ{orange, very thick}{0}{-1}{5}{0}{4};
    \filldraw[black] (1,4,-2) circle (1pt) node[above]{$B_1$};
    \filldraw[black] (2,4,-1) circle (1pt) node[above] {$B_2$};
    \filldraw[black] (3,4,-1) circle (1pt) node[above] {$B_3$};
    \filldraw[black] (4,4,0) circle (1pt) node[above] {$B_4$};
    \filldraw[black] (0,0,-4) circle (1pt) node[below] {$A_1$};
    \filldraw[black] (1,0,-3) circle (1pt) node[below] {$A_2$};
    \filldraw[black] (2,0,-3) circle (1pt) node[below] {$A_3$};
    \filldraw[black] (3,0,-2) circle (1pt) node[below] {$A_4$};
    
    \draw [->, blue, very thick] (0, 0, -4) -- (-1, 0, -3);
    \draw [->, blue, very thick] (-1, 0, -3) -- (-1, 0, -2) node[pos=0.4, above]{$t_3$};
    
    \draw [->, black, very thick] (-1, 0, -2) -- (-1, 1, -2);
    \draw [->, black, very thick] (-1, 1, -2) -- (-1, 1, -2);
    \draw [->, black, very thick] (-1, 1, -2) -- (-1, 2, -2);
    \draw [->, black, very thick] (-1, 2, -2) -- (0, 3, -2) node[pos=0.6, left, black]{$x_2$};
    \draw [->, black, very thick] (0, 3, -2) -- (1, 4, -2) node[pos=0.6, left, black]{$x_1$};
    \draw [->, blue, very thick] (1, 0, -3) -- (0, 0, -2);
    \draw [->, blue, very thick] (0, 0, -2) -- (-1, 0, -1);
    
    \draw [->, black, very thick] (-1, 0, -1) -- (-1, 1, -1);   
    \draw [->, black, very thick] (-1, 1, -1) -- (0, 2, -1) node[pos=0.6, left, black]{$x_3$};
    \draw [->, black, very thick] (0, 2, -1) -- (1, 3, -1) node[pos=0.6, left, black]{$x_2$};
    \draw [->, black, very thick] (1, 3, -1) -- (2, 4, -1) node[pos=0.6, left, black]{$x_1$};
    \draw [->, blue, very thick] (2, 0, -3) -- (2, 0, -2) node[pos=0.4, above]{$t_3$};
    \draw [->, blue, very thick] (2, 0, -2) -- (1, 0, -1);
    
    \draw [->, black, very thick] (1, 0, -1) -- (1, 1, -1);
    \draw [->, black, very thick] (1, 1, -1) -- (2, 2, -1) node[pos=0.6, left, black]{$x_3$};
    \draw [->, black, very thick] (2, 2, -1) -- (2, 3, -1);
    \draw [->, black, very thick] (2, 3, -1) -- (3, 4, -1) node[pos=0.6, left, black]{$x_1$};
    \draw [->, blue, very thick] (3, 0, -2) -- (3, 0, -1) node[pos=0.4, above]{$t_2$};
    \draw [->, blue, very thick] (3, 0, -1) -- (2, 0, 0);
    
    \draw [->, black, very thick] (2, 0, 0) -- (3, 1, 0) node[pos=0.6, left, black]{$x_4$};
    \draw [->, black, very thick] (3, 1, 0) -- (3, 2, 0);
    \draw [->, black, very thick] (3, 2, 0) -- (4, 3, 0) node[pos=0.6, left, black]{$x_2$};
    \draw [->, black, very thick] (4, 3, 0) -- (4, 4, 0);

    \end{tikzpicture}
}

\newcommand{\pictureslideA}{
\begin{tikzpicture}[x=1.2cm,y=1cm,z=0.5cm,rotate around y=33,rotate around x=0]\label{slide1}

    \draw[->, dashed, lightgray] (xyz cs:x=-3) -- (xyz cs:x=5.3) node[above] {$x$};
    \draw[->, dashed, lightgray] (xyz cs:y=0) -- (xyz cs:y=4.3) node[right] {$y$};
    \draw[->, dashed, lightgray] (xyz cs:z=0) -- (xyz cs:z=-4.3) node[above] {$z$};
    \showAllXZ{0}{4}{-4}{ultra thin, dashed, lightgray};

    \drawrectZ{orange, very thick}{0}{-1}{5}{0}{4};
    \filldraw[black] (1,4,-2) circle (1pt) node[above]{$B_1$};
    \filldraw[black] (2,4,-1) circle (1pt) node[above] {$B_2$};
    \filldraw[black] (3,4,-1) circle (1pt) node[above] {$B_3$};
    \filldraw[black] (4,4,0) circle (1pt) node[above] {$B_4$};
    \filldraw[black] (0,0,-4) circle (1pt) node[below] {$A_1$};
    \filldraw[black] (1,0,-3) circle (1pt) node[below] {$A_2$};
    \filldraw[black] (2,0,-3) circle (1pt) node[below] {$A_3$};
    \filldraw[black] (3,0,-2) circle (1pt) node[below] {$A_4$};

    \draw [->, blue, very thick] (3, 0, -2) -- (3, 0, -1);
    \draw [->, blue, very thick] (3, 0, -1) -- (2, 0, 0);
    
    \draw [->, black, very thick] (2, 0, 0) -- (3, 1, 0);
    \draw [->, black, very thick] (3, 1, 0) -- (3, 2, 0);
    \draw [->, black, very thick] (3, 2, 0) -- (4, 3, 0);
    \draw [->, black, very thick] (4, 3, 0) -- (4, 4, 0);

    \end{tikzpicture}
}

\newcommand{\pictureslideB}{
\begin{tikzpicture}[x=1.2cm,y=1cm,z=0.5cm,rotate around y=33,rotate around x=0]\label{slide1}
    \draw[->, dashed, lightgray] (xyz cs:x=-3) -- (xyz cs:x=5.3) node[above] {$x$};
    \draw[->, dashed, lightgray] (xyz cs:y=0) -- (xyz cs:y=4.3) node[right] {$y$};
    \draw[->, dashed, lightgray] (xyz cs:z=0) -- (xyz cs:z=-4.3) node[above] {$z$};
    \showAllXZ{0}{4}{-4}{ultra thin, dashed, lightgray};

    \drawrectZ{gray, thin}{0}{-1}{5}{0}{4}

    \drawrectZ{orange, very thick}{-1}{-1}{5}{0}{4};

    \filldraw[black] (1,4,-2) circle (1pt) node[above]{$B_1$};
    \filldraw[black] (2,4,-1) circle (1pt) node[above] {$B_2$};
    \filldraw[black] (3,4,-1) circle (1pt) node[above] {$B_3$};
    \filldraw[black] (4,4,0) circle (1pt) node[above] {$B_4$};
    \filldraw[black] (0,0,-4) circle (1pt) node[below] {$A_1$};
    \filldraw[black] (1,0,-3) circle (1pt) node[below] {$A_2$};
    \filldraw[black] (2,0,-3) circle (1pt) node[below] {$A_3$};
    \filldraw[black] (3,0,-2) circle (1pt) node[below] {$A_4$};
    
    \draw [->, blue, very thick] (1, 0, -3) -- (0, 0, -2);
    \draw [->, blue, very thick] (0, 0, -2) -- (-1, 0, -1);
    
    \draw [->, black, very thick] (-1, 0, -1) -- (-1, 1, -1);   
    \draw [->, black, very thick] (-1, 1, -1) -- (0, 2, -1);
    \draw [->, black, very thick] (0, 2, -1) -- (1, 3, -1);
    \draw [->, black, very thick] (1, 3, -1) -- (2, 4, -1);
    \draw [->, blue, very thick] (2, 0, -3) -- (2, 0, -2);
    \draw [->, blue, very thick] (2, 0, -2) -- (1, 0, -1);
    
    \draw [->, black, very thick] (1, 0, -1) -- (1, 1, -1);
    \draw [->, black, very thick] (1, 1, -1) -- (2, 2, -1);
    \draw [->, black, very thick] (2, 2, -1) -- (2, 3, -1);
    \draw [->, black, very thick] (2, 3, -1) -- (3, 4, -1);
    
    \draw [->, blue, very thick] (3, 0, -2) -- (3, 0, -1);
    \draw [->, blue, dashed] (3, 0, -1) -- (2, 0, 0);
    \draw [->, black, very thick] (3, 0, -1) -- (4, 1, -1);
    \draw [->, black, very thick] (4, 1, -1) -- (4, 2, -1);
    \draw [->, black, very thick, dashed] (4, 2, -1) -- (4, 4, -1);
    \draw [->, orange, very thick] (4, 2, -1) -- (4, 3, 0) node[pos=0.6, right, black]{$x_2$};
    \draw [->, orange, very thick] (4, 3, 0) -- (4, 4, 0);
    \end{tikzpicture}
}

\newcommand{\pictureslideC}{
\begin{tikzpicture}[x=1.2cm,y=1cm,z=0.5cm,rotate around y=33,rotate around x=0]\label{slide1}
    \draw[->, dashed, lightgray] (xyz cs:x=-3) -- (xyz cs:x=5.3) node[above] {$x$};
    \draw[->, dashed, lightgray] (xyz cs:y=0) -- (xyz cs:y=4.3) node[right] {$y$};
    \draw[->, dashed, lightgray] (xyz cs:z=0) -- (xyz cs:z=-4.3) node[above] {$z$};
    \showAllXZ{0}{4}{-4}{ultra thin, dashed, lightgray};
    \drawrectZ{gray, thin}{0}{-1}{5}{0}{4}
    \drawrectZ{gray, thin}{-1}{-1}{5}{0}{4}
    \drawrectZ{orange, very thick}{-2}{-1}{5}{0}{4};

    \filldraw[black] (1,4,-2) circle (1pt) node[above]{$B_1$};
    \filldraw[black] (2,4,-1) circle (1pt) node[above] {$B_2$};
    \filldraw[black] (3,4,-1) circle (1pt) node[above] {$B_3$};
    \filldraw[black] (4,4,0) circle (1pt) node[above] {$B_4$};
    \filldraw[black] (0,0,-4) circle (1pt) node[below] {$A_1$};
    \filldraw[black] (1,0,-3) circle (1pt) node[below] {$A_2$};
    \filldraw[black] (2,0,-3) circle (1pt) node[below] {$A_3$};
    \filldraw[black] (3,0,-2) circle (1pt) node[below] {$A_4$};

    \draw [->, blue, very thick] (0, 0, -4) -- (-1, 0, -3);
    \draw [->, blue, very thick] (-1, 0, -3) -- (-1, 0, -2);
    
    \draw [->, black, very thick] (-1, 0, -2) -- (-1, 1, -2);
    \draw [->, black, very thick] (-1, 1, -2) -- (-1, 1, -2);
    \draw [->, black, very thick] (-1, 1, -2) -- (-1, 2, -2);
    \draw [->, black, very thick] (-1, 2, -2) -- (0, 3, -2);
    \draw [->, black, very thick] (0, 3, -2) -- (1, 4, -2);
    \draw [->, blue, very thick] (1, 0, -3) -- (0, 0, -2);
    \draw [->, blue, very thick, dashed] (0, 0, -2) -- (-1, 0, -1);
    
    \draw [->, black, very thick] (0, 0, -2) -- (0, 1, -2);   
    \draw [->, black, very thick] (0, 1, -2) -- (1, 2, -2);
    \draw [->, black, very thick] (1, 2, -2) -- (2, 3, -2);
    \draw [->, orange, very thick] (2, 3, -2) -- (2, 4, -1) node[pos=0.6, right, black]{$x_1$};
    
    \draw [->, blue, very thick] (2, 0, -3) -- (2, 0, -2);
    \draw [->, blue, very thick, dashed] (2, 0, -2) -- (1, 0, -1);
    
    \draw [->, black, very thick] (2, 0, -2) -- (2, 1, -2);
    \draw [->, black, very thick] (2, 1, -2) -- (3, 2, -2);
    \draw [->, black, very thick] (3, 2, -2) -- (3, 3, -2);
    \draw [->, orange, very thick] (3, 3, -2) -- (3, 4, -1) node[pos=0.6, right, black]{$x_1$};
    
    \draw [->, gray, very thick, dashed] (1, 0, -3) -- (1, 0, -2);
    \draw [->, gray, very thick, dashed] (1, 0, -2) -- (1, 1, -2);   
    \draw [->, gray, very thick, dashed] (1, 1, -2) -- (2, 2, -2);
    \draw [->, gray, very thick, dashed] (2, 2, -2) -- (2, 3, -2);
    \draw [->, black, very thick, dashed] (2, 3, -2) -- (2, 4, -2);

    \draw [->, gray, very thick, dotted] (2, 0, -3) -- (1, 0, -2);
    \draw [->, gray, very thick, dotted] (1, 0, -2) -- (1, 1, -2);
    \draw [->, gray, very thick, dotted] (1, 1, -2) -- (2, 2, -2);
    \draw [->, gray, very thick, dotted] (2, 2, -2) -- (3, 3, -2);
    \draw [->, black, very thick, dashed] (3, 3, -2) -- (3, 4, -2);
    
    \draw [->, blue, dashed] (3, 0, -2) -- (3, 0, -1);
    \draw [->, blue, dashed] (3, 0, -1) -- (2, 0, 0);
    \draw [->, black, very thick] (3, 0, -2) -- (4, 1, -2);
    \draw [->, black, very thick, dashed] (4, 1, -2) -- (4, 4, -2);
    \draw [->, orange, very thick] (4, 1, -2) -- (4, 1, -1) node[pos=0.9, below, black]{$t_2$};
    \draw [->, orange, very thick] (4, 1, -1) -- (4, 2, -1);
    \draw [->, orange, very thick] (4, 2, -1) -- (4, 3, 0) node[pos=0.6, right, black]{$x_2$};
    \draw [->, orange, very thick] (4, 3, 0) -- (4, 4, 0);
    \end{tikzpicture}
}

\newcommand{\pictureslideD}{
\begin{tikzpicture}[x=1.2cm,y=1cm,z=0.5cm,rotate around y=33,rotate around x=0]\label{slide1}
    \draw[->, dashed, lightgray] (xyz cs:x=-3) -- (xyz cs:x=5.3) node[above] {$x$};
    \draw[->, dashed, lightgray] (xyz cs:y=0) -- (xyz cs:y=4.3) node[right] {$y$};
    \draw[->, dashed, lightgray] (xyz cs:z=0) -- (xyz cs:z=-4.3) node[above] {$z$};
    \showAllXZ{0}{4}{-4}{ultra thin, dashed, lightgray};
    
    \drawrectZ{gray, thin}{0}{-1}{5}{0}{4}
    \drawrectZ{gray, thin}{-1}{-1}{5}{0}{4}
    \drawrectZ{gray, thin}{-2}{-1}{5}{0}{4}
    \drawrectZ{orange, very thick}{-3}{-1}{5}{0}{4};

    \filldraw[black] (1,4,-2) circle (1pt) node[above]{$B_1$};
    \filldraw[black] (2,4,-1) circle (1pt) node[above] {$B_2$};
    \filldraw[black] (3,4,-1) circle (1pt) node[above] {$B_3$};
    \filldraw[black] (4,4,0) circle (1pt) node[above] {$B_4$};
    \filldraw[black] (0,0,-4) circle (1pt) node[below] {$A_1$};
    \filldraw[black] (1,0,-3) circle (1pt) node[below] {$A_2$};
    \filldraw[black] (2,0,-3) circle (1pt) node[below] {$A_3$};
    \filldraw[black] (4,0,-3) circle (1pt) node[below] {$F_4$};

    \draw [->, blue, very thick] (0, 0, -4) -- (-1, 0, -3);
    \draw [->, blue, dashed, thick] (-1, 0, -3) -- (-1, 0, -2);
    
    \draw [->, black, very thick] (-1, 0, -3) -- (-1, 1, -3);
    \draw [->, black, very thick] (-1, 1, -3) -- (-1, 1, -3);
    \draw [->, black, very thick] (-1, 1, -3) -- (-1, 2, -3);
    \draw [->, black, very thick] (-1, 2, -3) -- (0, 3, -3);
    \draw [->, black, very thick] (0, 3, -3) -- (1, 4, -3);
    \draw [->, orange, very thick] (1, 4, -3) -- (1, 4, -2) node[pos=0.7, below, black]{$t_3$};
    \draw [->, blue, dashed, thick] (1, 0, -3) -- (0, 0, -2);
    \draw [->, blue, dashed, thick] (0, 0, -2) -- (0, 0, -1);
    \draw [->, black, very thick] (1, 0, -3) -- (1, 1, -3);
    \draw [->, black, very thick] (1, 1, -3) -- (2, 2, -3);
    
    \draw [->, black, very thick, dashed] (2, 2, -3) -- (2, 4, -3);
    \draw [->, orange, very thick] (2, 2, -3) -- (2, 2, -2) node[pos=0.4, below, black]{$t_3$};
    \draw [->, orange, very thick] (2, 2, -2) -- (2, 3, -2);
    \draw [->, orange, very thick] (2, 3, -2) -- (2, 4, -1) node[pos=0.6, right, black]{$x_1$};
    \draw [->, blue, thick, dashed] (2, 0, -3) -- (1, 0, -2);
    \draw [->, blue, thick, dashed] (1, 0, -2) -- (0, 0, -1);
    \draw [->, black, very thick] (2, 0, -3) -- (3, 1, -3);   
    \draw [->, black, very thick, dashed] (3, 1, -3) -- (3, 4, -3);

    \draw [->, orange, very thick] (3, 1, -3) -- (3, 2, -2) node[pos=0.6, right, black]{$x_3$};
    \draw [->, orange, very thick] (3, 2, -2) -- (3, 3, -2);
    \draw [->, orange, very thick] (3, 3, -2) -- (3, 4, -1) node[pos=0.6, right, black]{$x_1$};
    \draw [->, black, thick, dashed] (4, 0, -3) -- (3, 0, -2);
    \draw [->, blue, dashed] (3, 0, -2) -- (3, 0, -1);
    \draw [->, blue, dashed] (3, 0, -1) -- (2, 0, 0);
    \draw [->, orange, very thick] (4, 0, -3) -- (4, 1, -2) node[pos=0.6, right, black]{$x_4$};
    \draw [->, orange, very thick] (4, 1, -2) -- (4, 1, -1) node[pos=0.9, below, black]{$t_2$};
    \draw [->, orange, very thick] (4, 1, -1) -- (4, 2, -1);
    \draw [->, orange, very thick] (4, 2, -1) -- (4, 3, 0) node[pos=0.6, right, black]{$x_2$};
    \draw [->, orange, very thick] (4, 3, 0) -- (4, 4, 0);

    \end{tikzpicture}
}

\newcommand{\pictureslideE}{
\begin{tikzpicture}[x=1.2cm,y=1cm,z=0.5cm,rotate around y=33,rotate around x=0]\label{slide1}
    \draw[->, dashed, lightgray] (xyz cs:x=-3) -- (xyz cs:x=5.3) node[above] {$x$};
    \draw[->, dashed, lightgray] (xyz cs:y=0) -- (xyz cs:y=4.3) node[right] {$y$};
    \draw[->, dashed, lightgray] (xyz cs:z=0) -- (xyz cs:z=-4.3) node[above] {$z$};
    \showAllXZ{0}{4}{-4}{ultra thin, dashed, lightgray};
    
    \drawrectZ{gray, thin}{0}{-1}{5}{0}{4}
    \drawrectZ{gray, thin}{-1}{-1}{5}{0}{4}
    \drawrectZ{gray, thin}{-2}{-1}{5}{0}{4}
    \drawrectZ{gray, thin}{-3}{-1}{5}{0}{4};
    \drawrectZ{gray, thin}{-4}{-1}{5}{0}{4};
    \drawrectZ{orange, very thick}{-5}{-1}{5}{0}{4};

    \filldraw[black] (1,4,-2) circle (1pt) node[above]{$B_1$};
    \filldraw[black] (2,4,-1) circle (1pt) node[above] {$B_2$};
    \filldraw[black] (3,4,-1) circle (1pt) node[above] {$B_3$};
    \filldraw[black] (4,4,0) circle (1pt) node[above] {$B_4$};

    \filldraw[black] (1,0,-5) circle (1pt) node[below] {$F_1$};
    \filldraw[black] (2,0,-4) circle (1pt) node[below] {$F_2$};
    \filldraw[black] (3,0,-4) circle (1pt) node[below] {$F_3$};
    \filldraw[black] (4,0,-3) circle (1pt) node[below] {$F_4$};

    \draw [->, black, dashed, thick] (1, 0, -5) -- (0, 0, -4);
    \draw [->, blue, dashed, thick] (0, 0, -4) -- (-1, 0, -3);
    \draw [->, blue, dashed, thick] (-1, 0, -3) -- (-1, 0, -2);
    
    \draw [->, black, dashed, thick] (0, 0, -4) -- (0, 1, -4);
    \draw [->, black, dashed, thick] (0, 1, -4) -- (0, 1, -4);
    \draw [->, black, dashed, thick] (0, 1, -4) -- (0, 2, -4);
    \draw [->, black, dashed, thick] (0, 2, -4) -- (1, 3, -4);
    
    \draw [->, orange, very thick] (1, 0, -5) -- (1, 1, -5);
    \draw [->, orange, very thick] (1, 1, -5) -- (1, 2, -5);
    \draw [->, orange, very thick] (1, 2, -5) -- (1, 3, -4) node[pos=0.6, right, black]{$x_2$};
    \draw [->, orange, very thick] (1, 3, -4) -- (1, 4, -3) node[pos=0.6, right, black]{$x_1$};
    \draw [->, orange, very thick] (1, 4, -3) -- (1, 4, -2) node[pos=0.7, below, black]{$t_3$};
    \draw [->, black, dashed, thick] (2, 0, -4) -- (1, 0, -3);
    \draw [->, blue, dashed, thick] (1, 0, -3) -- (0, 0, -2);
    \draw [->, blue, dashed, thick] (0, 0, -2) -- (0, 0, -1);
    \draw [->, orange, very thick] (2, 0, -4) -- (2, 1, -4);
    \draw [->, orange, very thick] (2, 1, -4) -- (2, 2, -3)  node[pos=0.6, right, black]{$x_3$};
    \draw [->, orange, very thick] (2, 2, -3) -- (2, 2, -2) node[pos=0.4, below, black]{$t_3$};
    \draw [->, orange, very thick] (2, 2, -2) -- (2, 3, -2);
    \draw [->, orange, very thick] (2, 3, -2) -- (2, 4, -1) node[pos=0.6, right, black]{$x_1$};
    \draw [->, blue, thick, dashed] (3, 0, -4) -- (2, 0, -3);
    \draw [->, blue, thick, dashed] (2, 0, -3) -- (1, 0, -2);
    \draw [->, blue, thick, dashed] (1, 0, -2) -- (0, 0, -1);
    \draw [->, orange, very thick] (3, 0, -4) -- (3, 1, -3) node[pos=0.6, right, black]{$x_4$};
    \draw [->, orange, very thick] (3, 1, -3) -- (3, 2, -2) node[pos=0.6, right, black]{$x_3$};
    \draw [->, orange, very thick] (3, 2, -2) -- (3, 3, -2);
    \draw [->, orange, very thick] (3, 3, -2) -- (3, 4, -1) node[pos=0.6, right, black]{$x_1$};
    \draw [->, black, thick, dashed] (4, 0, -3) -- (3, 0, -2);
    \draw [->, blue, dashed] (3, 0, -2) -- (3, 0, -1);
    \draw [->, blue, dashed] (3, 0, -1) -- (2, 0, 0);
    \draw [->, orange, very thick] (4, 0, -3) -- (4, 1, -2) node[pos=0.6, right, black]{$x_4$};
    \draw [->, orange, very thick] (4, 1, -2) -- (4, 1, -1) node[pos=0.9, below, black]{$t_2$};
    \draw [->, orange, very thick] (4, 1, -1) -- (4, 2, -1);
    \draw [->, orange, very thick] (4, 2, -1) -- (4, 3, 0) node[pos=0.6, right, black]{$x_2$};
    \draw [->, orange, very thick] (4, 3, 0) -- (4, 4, 0);

    \end{tikzpicture}
}

\newcommand{\picturewavessetup}{
    \draw[->, dashed, lightgray, very thin] (xyz cs:x=0) -- (xyz cs:x=6) node[above] {$x$};
    \draw[->, dashed, lightgray, very thin] (xyz cs:y=0) -- (xyz cs:y=4) node[right] {$y$};
    \draw[->, dashed, lightgray, very thin] (xyz cs:z=0) -- (xyz cs:z=-3) node[above] {$z$};
    \drawpoints{0}{6}{0}{0}{-1}{-2}{black}
}
\newcommand{\picturewavesA}{
\begin{tikzpicture}[x=1cm,y=1cm,z=0.4cm,rotate around y=0,rotate around x=0,rotate around z = 0]\label{pictureslidek}
    \picturewavessetup
    
    \drawplaneZ{lightgray, very thick}{0}{6}{4}{-1}
    \node [fill=white] at (0.5,3.5,-1) {$z = k$};

    \draw [-, very thick, blue] (1, 0, -2) -- (0, 0, -1);
    \draw [-, very thick, black]    (0, 0, -1)
                                    --(1, 1, -1)
                                    --(1, 2, -1)
                                    --(2, 3, -1);
    \draw [-, very thick, red] (2, 3, -1) -- (3, 4, -1);
    
    \draw [-, very thick, blue] (2, 0, -2) -- (1, 0, -1);
    \draw [-, very thick, black]    (1, 0, -1)
                                    --(2, 1, -1)
                                    --(2, 2, -1)
                                    --(3, 3, -1);
    \draw [-, very thick, red] (3, 3, -1) -- (4, 4, -1);
    
    \draw [-, thick, black, dashed]    (2, 0, -2)
                                    --(2, 0, -1)
                                    --(2, 1, -1);
    \draw [-, thick, black, dashed]    (3, 0, -2)
                                    --(2, 0, -1)
                                    --(3, 1, -1);

    \draw [-, very thick, blue] (3, 0, -2) -- (3, 0, -1);
    \draw [-, very thick, black]    (3, 0, -1)
                                    --(3, 1, -1)
                                    --(4, 2, -1)
                                    --(4, 3, -1)
                                    --(5, 4, -1);
    \draw [-, very thick, blue] (4, 0, -2) -- (4, 0, -1);
    \draw [-, very thick, black]    (4, 0, -1)
                                    --(5, 1, -1)
                                    --(5, 2, -1)
                                    --(6, 3, -1)
                                    --(6, 4, -1);
    
    \filldraw[black] (6, 4, -1) circle (1pt) node(B1)[anchor=south]{$B_{4}$};
    \filldraw[black] (5, 4, -1) circle (1pt) node(B2)[anchor=south]{$B_{3}$};
    \filldraw[black] (4, 4, -1) circle (1pt) node(B2)[anchor=south]{$B_{2}$};
    \filldraw[black] (3, 4, -1) circle (1pt) node(B2)[anchor=south]{$B_{1}$};
\end{tikzpicture}
}

\newcommand{\picturewavesAinv}{
\begin{tikzpicture}[x=1cm,y=1cm,z=0.4cm,rotate around y=0,rotate around x=0,rotate around z = 0]\label{pictureslidek}
    \picturewavessetup

    \node [fill=white] at (0.5,3.5,-2) {$z = k + 1$};
    \drawplaneZ{lightgray, very thick}{0}{6}{4}{-2}
    \draw [-, very thick, blue] (1, 0, -2) -- (0, 0, -1);
    \draw [-, very thick, black]    (1, 0, -2)
                                    --(1.95, 1, -2)
                                    --(1.95, 2, -2)
                                    --(3, 3.05, -2)
                                    --(3, 4, -2);
    \draw [-, very thick, blue] (2, 0, -2) -- (1, 0, -1);
    \draw [-, very thick, black]    (1.95, 0, -2)
                                    --(2.95, 1, -2)
                                    --(3, 2, -2)
                                    --(3.95, 3, -2)
                                    --(4, 4, -2);
    
    \draw [-, thick, black, dashed] (2.05, 0, -2)
                                    -- (2.05, 1, -2)
                                    -- (2.05, 2, -2)
                                    -- (3.05, 3, -2)
                                    -- (4.00, 4, -2);
    \draw [-, thick, black, dashed] (3, 0, -2)
                                    -- (3.95, 1, -2)
                                    -- (4.95, 2, -2)
                                    -- (5, 3, -2)
                                    -- (5, 4, -2);

    \draw [-, very thick, blue] (3, 0, -2) -- (3, 0, -1);
    \draw [-, very thick, black]    (3, 0, -2)
                                    --(3.05, 1, -2)
                                    --(4, 2, -2)
                                    --(4.05, 3, -2)
                                    --(5, 4, -2);
    \draw [-, very thick, blue] (4, 0, -2) -- (4, 0, -1);
    \draw [-, very thick, black]    (4, 0, -2)
                                    --(5, 1, -2)
                                    --(5, 1.95, -2)
                                    --(6, 2.95, -2)
                                    --(6, 4, -2);

    \filldraw[black] (3, 1, -2) circle (3pt);
    \filldraw[black] (6, 4, -2) circle (1pt) node(B1)[anchor=south]{$B_{4}$};
    \filldraw[black] (5, 4, -2) circle (1pt) node(B2)[anchor=south]{$B_{3}$};
    \filldraw[black] (4, 4, -2) circle (1pt) node(B2)[anchor=south]{$B_{2}$};
    \filldraw[black] (3, 4, -2) circle (1pt) node(B2)[anchor=south]{$B_{1}$};
\end{tikzpicture}
}

\newcommand{\picturewavesB}{
\begin{tikzpicture}[x=1cm,y=1cm,z=0.4cm,rotate around y=0,rotate around x=0,rotate around z = 0]\label{pictureslidek}
    \picturewavessetup
    
    \drawplaneZ{lightgray, very thick}{0}{6}{4}{-1}
    \node [fill=white] at (0.5,3.5,-1) {$z = k$};
    
    \draw [-, very thick, blue] (1, 0, -2) -- (0, 0, -1);
    \draw [-, very thick, blue] (2, 0, -2) -- (1.95, 0, -1);
    \draw [-, very thick, blue] (3, 0, -2) -- (2.05, 0, -1);
    \draw [-, very thick, blue] (4, 0, -2) -- (4, 0, -1);
    \draw [-, very thick, black]    (0, 0, -1)
                                    --(1, 1, -1)
                                    --(1, 2, -1)
                                    --(2, 3, -1);
    \draw [-, very thick, red] (2, 3, -1) -- (3, 4, -1);
    
    \draw [-, thick, black, dashed] (1, 0, -2)
                                    -- (1, 0, -1)
                                    -- (1, 1, -1);
    \draw [-, thick, black, dashed] (2, 0, -2)
                                    -- (1, 0, -1)
                                    -- (2, 1, -1);
    
    \draw [-, very thick, black]    (1.95, 0, -1)
                                    --(2, 1, -1)
                                    --(2, 2, -1)
                                    --(3, 3, -1)
                                    --(4, 4, -1);
    \draw [-, very thick, black]    (2.05, 0, -1)
                                    --(3, 1, -1)
                                    --(4, 2, -1)
                                    --(4, 3, -1);
    \draw [-, very thick, red] (4, 3, -1) -- (5, 4, -1);
    
    \draw [-, very thick, black]    (4, 0, -1)
                                    --(5, 1, -1)
                                    --(5, 2, -1)
                                    --(6, 3, -1)
                                    --(6, 4, -1);
    
    \filldraw[black] (6, 4, -1) circle (1pt) node(B1)[anchor=south]{$B_{4}$};
    \filldraw[black] (5, 4, -1) circle (1pt) node(B2)[anchor=south]{$B_{3}$};
    \filldraw[black] (4, 4, -1) circle (1pt) node(B2)[anchor=south]{$B_{2}$};
    \filldraw[black] (3, 4, -1) circle (1pt) node(B2)[anchor=south]{$B_{1}$};
\end{tikzpicture}
}

\newcommand{\picturewavesBinv}{
\begin{tikzpicture}[x=1cm,y=1cm,z=0.4cm,rotate around y=0,rotate around x=0, rotate around z = 0]\label{pictureslidek}
    \picturewavessetup

    \node [fill=white] at (0.5,3.5,-2) {$z = k + 1$};
    \drawplaneZ{lightgray, very thick}{0}{6}{4}{-2}
    
    \draw [-, very thick, blue] (1, 0, -2) -- (0, 0, -1);
    \draw [-, very thick, blue] (2, 0, -2) -- (1.95, 0, -1);
    \draw [-, very thick, blue] (3, 0, -2) -- (2.05, 0, -1);
    \draw [-, very thick, blue] (4, 0, -2) -- (4, 0, -1);
    \draw [-, very thick, black]  (1, 0, -2)
                                --(1.95, 1, -2)
                                --(1.95, 2, -2)
                                --(3, 3.05, -2)
                                --(3, 4, -2);

    \draw [-, thick, black, dashed] (1, 0, -2)
                                    --(1, 1, -2)
                                    --(1, 2, -2)
                                    --(2, 3, -2)
                                    --(3, 4, -2);
    \draw [-, thick, black, dashed] (2, 0, -2)
                                    -- (3, 1, -2)
                                    -- (3, 2, -2)
                                    -- (4, 3, -2)
                                    -- (4, 4, -2);

    \draw [-, very thick, black] (2.05, 0, -2)
                                    -- (2.05, 1, -2)
                                    -- (2.05, 2, -2)
                                    -- (3.05, 3, -2)
                                    -- (4.00, 4, -2);

    \draw [-, very thick, black] (3, 0, -2)
                                    -- (3.95, 1, -2)
                                    -- (5, 2.05, -2)
                                    -- (5, 3, -2)
                                    -- (5.00, 4, -2);

    \draw [-, very thick, black]    (4, 0, -2)
                                    --(5.05, 1, -2)
                                    --(5.05, 2, -2)
                                    --(6, 3, -2)
                                    --(6, 4, -2);

    \filldraw[black] (2, 1, -2) circle (3pt);
    \filldraw[black] (6, 4, -2) circle (1pt) node(B1)[anchor=south]{$B_{4}$};
    \filldraw[black] (5, 4, -2) circle (1pt) node(B2)[anchor=south]{$B_{3}$};
    \filldraw[black] (4, 4, -2) circle (1pt) node(B2)[anchor=south]{$B_{2}$};
    \filldraw[black] (3, 4, -2) circle (1pt) node(B2)[anchor=south]{$B_{1}$};
\end{tikzpicture}
}

\newcommand{\picturewavesC}{
\begin{tikzpicture}[x=1cm,y=1cm,z=0.4cm,rotate around y=0,rotate around x=0,rotate around z = 0]\label{pictureslidek}
    \picturewavessetup
    
    \drawplaneZ{lightgray, very thick}{0}{6}{4}{-1}
    \node [fill=white] at (0.5,3.5,-1) {$z = k$};
    \draw [-, very thick, blue] (1, 0, -2) -- (1, 0, -1);
    \draw [-, very thick, blue] (2, 0, -2) -- (1, 0, -1);
    \draw [-, very thick, blue] (3, 0, -2) -- (2, 0, -1);
    \draw [-, very thick, blue] (4, 0, -2) -- (4, 0, -1);
    
    \draw [-, very thick, black]    (1, 0, -1)
                                    --(1, 1, -1)
                                    --(1, 2, -1)
                                    --(2, 3, -1)
                                    --(3, 4, -1);
    
    \draw [-, very thick, black]    (1, 0, -1)
                                    --(2, 1, -1)
                                    --(2, 2, -1)
                                    --(3, 3, -1);
    \draw [-, very thick, red] (3, 3, -1) -- (4, 4, -1);
    
    \draw [-, very thick, black]    (2.05, 0, -1)
                                    --(3, 1, -1)
                                    --(4, 2, -1)
                                    --(4, 3, -1);
    \draw [-, very thick, red] (4, 3, -1) -- (5, 4, -1);
    
    \draw [-, thick, black, dashed] (3, 0, -2)
                                    -- (3, 0, -1)
                                    -- (4, 1, -1)
                                    -- (4, 2, -1);
    \draw [-, thick, black, dashed] (4, 0, -2)
                                    -- (3, 0, -1)
                                    -- (4, 1, -1)
                                    -- (5, 2, -1);
    
    \draw [-, very thick, black]    (4, 0, -1)
                                    --(5, 1, -1)
                                    --(5, 2, -1)
                                    --(6, 3, -1)
                                    --(6, 4, -1);
    
    \filldraw[black] (6, 4, -1) circle (1pt) node(B1)[anchor=south]{$B_{4}$};
    \filldraw[black] (5, 4, -1) circle (1pt) node(B2)[anchor=south]{$B_{3}$};
    \filldraw[black] (4, 4, -1) circle (1pt) node(B2)[anchor=south]{$B_{2}$};
    \filldraw[black] (3, 4, -1) circle (1pt) node(B2)[anchor=south]{$B_{1}$};
\end{tikzpicture}
}

\newcommand{\picturewavesCinv}{
\begin{tikzpicture}[x=1cm,y=1cm,z=0.4cm,rotate around y=0,rotate around x=0, rotate around z = 0]\label{pictureslidek}
    \picturewavessetup

    \node [fill=white] at (0.5,3.5,-2) {$z = k + 1$};
    \drawplaneZ{lightgray, very thick}{0}{6}{4}{-2}
    \draw [-, very thick, blue] (1, 0, -2) -- (1, 0, -1);
    \draw [-, very thick, blue] (2, 0, -2) -- (1, 0, -1);
    \draw [-, very thick, blue] (3, 0, -2) -- (2, 0, -1);
    \draw [-, very thick, blue] (4, 0, -2) -- (4, 0, -1);
    \draw [-, very thick, black]  (1, 0, -2)
                                    --(1, 1, -2)
                                    --(1, 2, -2)
                                    --(2, 3, -2)
                                    --(3, 4, -2);

    \draw [-, very thick, black] (2, 0, -2)
                                    -- (3, 1, -2)
                                    -- (3, 2, -2)
                                    -- (4, 3, -2)
                                    -- (4, 4, -2);

    \draw [-, very thick, black] (3, 0, -2)
                                    -- (4, 1, -2)
                                    -- (5, 2.1, -2)
                                    -- (5, 3, -2)
                                    -- (5.00, 4, -2);

    \draw [-, dashed, thick, black](4, 1, -2)
                                -- (4, 2, -2)
                                -- (4, 3, -2)
                                -- (5.00, 4, -2);
    \draw [-, dashed, thick, black](5, 1, -2)
                                -- (6, 2, -2)
                                -- (6, 3, -2);

    \draw [-, very thick, black]    (4, 0, -2)
                                    --(5.05, 1, -2)
                                    --(5.05, 2, -2)
                                    --(6, 3, -2)
                                    --(6, 4, -2);

    \filldraw[black] (5, 2, -2) circle (3pt);
    \filldraw[black] (6, 4, -2) circle (1pt) node(B1)[anchor=south]{$B_{4}$};
    \filldraw[black] (5, 4, -2) circle (1pt) node(B2)[anchor=south]{$B_{3}$};
    \filldraw[black] (4, 4, -2) circle (1pt) node(B2)[anchor=south]{$B_{2}$};
    \filldraw[black] (3, 4, -2) circle (1pt) node(B2)[anchor=south]{$B_{1}$};
\end{tikzpicture}
}

\newcommand{\picturewavesD}{
\begin{tikzpicture}[x=1cm,y=1cm,z=0.4cm,rotate around y=0,rotate around x=0,rotate around z = 0]\label{pictureslidek}
    \picturewavessetup
    
    \drawplaneZ{lightgray, very thick}{0}{6}{4}{-1}
    \node [fill=white] at (0.5,3.5,-1) {$z = k$};
    \draw [-, very thick, blue] (1, 0, -2) -- (1, 0, -1);
    \draw [-, very thick, blue] (2, 0, -2) -- (1, 0, -1);
    \draw [-, very thick, blue] (3, 0, -2) -- (2.95, 0, -1);
    \draw [-, very thick, blue] (4, 0, -2) -- (3.05, 0, -1);
    
    \draw [-, very thick, black]    (1, 0, -1)
                                    --(1, 1, -1)
                                    --(1, 2, -1)
                                    --(2, 3, -1)
                                    --(3, 4, -1);
    
    \draw [-, very thick, black]    (1, 0, -1)
                                    --(2, 1, -1)
                                    --(2, 2, -1)
                                    --(3, 3, -1);
    \draw [-, very thick, red] (3, 3, -1) -- (4, 4, -1);
    
    \draw [-, thick, black, dashed] (2, 0, -2)
                                    -- (2, 0, -1)
                                    -- (3, 1, -1)
                                    -- (3, 2, -1)
                                    -- (3, 3, -1);
    \draw [-, thick, black, dashed] (3, 0, -2)
                                    -- (2, 0, -1)
                                    -- (3, 1, -1)
                                    -- (3, 2, -1)
                                    -- (4, 3, -1);
    
    \draw [-, very thick, black]    (2.95, 0, -1)
                                    --(4, 1.05, -1)
                                    --(4, 2, -1)
                                    --(4, 3, -1)
                                    --(5, 4, -1);
    
    \draw [-, very thick, black]    (3.05, 0, -1)
                                    --(4.05, 1, -1)
                                    --(5, 2, -1);
    \draw [-, very thick, black]  (6, 3, -1) -- (6, 4, -1);
    \draw [-, very thick, red] (5, 2, -1) -- (6, 3, -1);
    
    \filldraw[black] (6, 4, -1) circle (1pt) node(B1)[anchor=south]{$B_{4}$};
    \filldraw[black] (5, 4, -1) circle (1pt) node(B2)[anchor=south]{$B_{3}$};
    \filldraw[black] (4, 4, -1) circle (1pt) node(B2)[anchor=south]{$B_{2}$};
    \filldraw[black] (3, 4, -1) circle (1pt) node(B2)[anchor=south]{$B_{1}$};
\end{tikzpicture}
}

\newcommand{\picturewavesDinv}{
\begin{tikzpicture}[x=1cm,y=1cm,z=0.4cm,rotate around y=0,rotate around x=0, rotate around z = 0]\label{pictureslidek}
    \picturewavessetup

    \node [fill=white] at (0.5,3.5,-2) {$z = k + 1$};
    \drawplaneZ{lightgray, very thick}{0}{6}{4}{-2}
    \draw [-, very thick, blue] (1, 0, -2) -- (1, 0, -1);
    \draw [-, very thick, blue] (2, 0, -2) -- (1, 0, -1);
    \draw [-, very thick, blue] (3, 0, -2) -- (2.95, 0, -1);
    \draw [-, very thick, blue] (4, 0, -2) -- (3.05, 0, -1);
    
    \draw [-, very thick, black]  (1, 0, -2)
                                    --(1, 1, -2)
                                    --(1, 2, -2)
                                    --(2, 3, -2)
                                    --(3, 4, -2);

    \draw [-, very thick, black] (2, 0, -2)
                                    -- (2.95, 1, -2)
                                    -- (2.95, 2, -2)
                                    -- (3.95, 3, -2)
                                    -- (4, 4, -2);

    \draw [-, thick, black, dashed] (3, 2, -2)
                                    -- (3, 3, -2)
                                    -- (4, 4, -2);
    \draw [-, thick, black, dashed] (4, 2, -2)
                                    -- (5, 3, -2)
                                    -- (5, 4, -2);

    \draw [-, very thick, black] (3.05, 0, -2)
                                    -- (4.05, 1, -2)
                                    -- (4.05, 2, -2)
                                    -- (4.05, 3, -2)
                                    -- (5.00, 4, -2);

    \draw [-, very thick, black]    (4, 0, -2)
                                    --(5, 1, -2)
                                    --(6, 2, -2)
                                    --(6, 3, -2)
                                    --(6, 4, -2);

    \filldraw[black] (4, 3, -2) circle (3pt);
    \filldraw[black] (6, 4, -2) circle (1pt) node(B1)[anchor=south]{$B_{4}$};
    \filldraw[black] (5, 4, -2) circle (1pt) node(B2)[anchor=south]{$B_{3}$};
    \filldraw[black] (4, 4, -2) circle (1pt) node(B2)[anchor=south]{$B_{2}$};
    \filldraw[black] (3, 4, -2) circle (1pt) node(B2)[anchor=south]{$B_{1}$};
\end{tikzpicture}
}

\newcommand{\picturewavesE}{
\begin{tikzpicture}[x=1cm,y=1cm,z=0.4cm,rotate around y=0,rotate around x=0,rotate around z = 0]\label{pictureslidek}
    \picturewavessetup
    
    \drawplaneZ{lightgray, very thick}{0}{6}{4}{-1}
    \node [fill=white] at (0.5,3.5,-1) {$z = k$};
    
    \draw [-, very thick, blue] (1, 0, -2) -- (1, 0, -1);
    \draw [-, very thick, blue] (2, 0, -2) -- (1.95, 0, -1);
    \draw [-, very thick, blue] (3, 0, -2) -- (2.05, 0, -1);
    \draw [-, very thick, blue] (4, 0, -2) -- (3.05, 0, -1);
    
    \draw [-, very thick, black]    (1, 0, -1)
                                    --(1, 1, -1)
                                    --(1, 2, -1)
                                    --(2, 3, -1)
                                    --(3, 4, -1);
    
    \draw [-, very thick, black]    (1.95, 0, -1)
                                    --(2.95, 1, -1)
                                    --(2.95, 2, -1)
                                    --(2.95, 3, -1)
                                    --(4, 4, -1);

    \draw [-, very thick, black]    (2.05, 0, -1)
                                    --(3.05, 1, -1)
                                    --(3.05, 2, -1)
                                    --(4, 3, -1);
    \draw [-, very thick, red] (4, 3, -1) -- (5, 4, -1);
    
    \draw [-, very thick, black]    (3, 0, -1)
                                    --(4, 1, -1)
                                    --(5, 2, -1);
    \draw [-, very thick, black] (6, 3, -1) -- (6, 4, -1);
    \draw [-, very thick, red] (5, 2, -1) -- (6, 3, -1);
    
    \filldraw[black] (6, 4, -1) circle (1pt) node(B1)[anchor=south]{$B_{4}$};
    \filldraw[black] (5, 4, -1) circle (1pt) node(B2)[anchor=south]{$B_{3}$};
    \filldraw[black] (4, 4, -1) circle (1pt) node(B2)[anchor=south]{$B_{2}$};
    \filldraw[black] (3, 4, -1) circle (1pt) node(B2)[anchor=south]{$B_{1}$};
\end{tikzpicture}
}

\newcommand{\picturewavesEinv}{
\begin{tikzpicture}[x=1cm,y=1cm,z=0.4cm,rotate around y=0,rotate around x=0, rotate around z = 0]\label{pictureslidek}
    \picturewavessetup

    \node [fill=white] at (0.5,3.5,-2) {$z = k + 1$};
    \drawplaneZ{lightgray, very thick}{0}{6}{4}{-2}
    \draw [-, very thick, blue] (1, 0, -2) -- (1, 0, -1);
    \draw [-, very thick, blue] (2, 0, -2) -- (2, 0, -1);
    \draw [-, very thick, blue] (3, 0, -2) -- (2, 0, -1);
    \draw [-, very thick, blue] (4, 0, -2) -- (3, 0, -1);
    
    \draw [-, very thick, black]  (1, 0, -2)
                                    --(1, 1, -2)
                                    --(1, 2, -2)
                                    --(2, 3, -2)
                                    --(3, 4, -2);

    \draw [-, very thick, black] (2, 0, -2)
                                    -- (3, 1, -2)
                                    -- (3, 2, -2)
                                    -- (3, 3, -2)
                                    -- (4, 4, -2);

    \draw [-, very thick, black] (3, 0, -2)
                                    -- (4, 1, -2)
                                    -- (4, 2, -2)
                                    -- (5, 3, -2)
                                    -- (5, 4, -2);

    \draw [-, very thick, black]    (4, 0, -2)
                                    --(5, 1, -2)
                                    --(6, 2, -2)
                                    --(6, 3, -2)
                                    --(6, 4, -2);

    \filldraw[black] (6, 4, -2) circle (1pt) node(B1)[anchor=south]{$B_{4}$};
    \filldraw[black] (5, 4, -2) circle (1pt) node(B2)[anchor=south]{$B_{3}$};
    \filldraw[black] (4, 4, -2) circle (1pt) node(B2)[anchor=south]{$B_{2}$};
    \filldraw[black] (3, 4, -2) circle (1pt) node(B2)[anchor=south]{$B_{1}$};
\end{tikzpicture}
}

\newcommand{\pictureInvolutionSetup} {
	\draw[->, dashed, lightgray] (xyz cs:x=-3) -- (xyz cs:x=4.3) node[above] {$x$};
    \draw[->, dashed, lightgray] (xyz cs:y=0) -- (xyz cs:y=4.3) node[right] {$y$};
    \draw[->, dashed, lightgray] (xyz cs:z=0) -- (xyz cs:z=-3) node[above] {$z$};
    \showAllXZ{0}{4}{-3}{ultra thin, dashed, lightgray};
    \filldraw[black] (2,4,-1) circle (1pt) node[above] {$B_1$};
    \filldraw[black] (3,4,0) circle (1pt) node[above] {$B_2$};
    \filldraw[black] (4,4,0) circle (1pt) node[above] {$B_3$};
    \filldraw[black] (1,0,-3) circle (1pt) node[below] {$A_1$};
    \filldraw[black] (2,0,-3) circle (1pt) node[below] {$A_2$};
    \filldraw[black] (3,0,-2) circle (1pt) node[below] {$A_3$};
}
\newcommand{\pictureInvolutionInit}{
\begin{tikzpicture}[x=1.2cm,y=1cm,z=0.5cm,rotate around y=33,rotate around x=0]\label{slide1}

    \pictureInvolutionSetup
    \drawrectZ{gray, thick}{-1}{-1}{4}{0}{4}
    \drawrectZ{gray, very thin}{0}{-1}{4}{0}{4};
    \draw [->, blue, very thick] (1, 0, -3) -- (0, 0, -2);
    \draw [->, blue, very thick] (0, 0, -2) -- (0, 0, -1);
    \draw [->, blue, thin] (0, 0, -1) -- (-1, 0, 0);
    
    \draw [->, black,  thin] (-1, 0, 0) -- (0, 1, 0);
    \draw [->, black,  thin] (0, 1, 0) -- (1, 2, 0);
    \draw [->, black,  thin] (1, 2, 0) -- (2, 3, 0);
    \draw [->, black,  thin] (2, 3, 0) -- (3, 4, 0);
    
    \draw [->, blue, very thick] (2, 0, -3) -- (2, 0, -2);
    \draw [->, blue, very thick] (2, 0, -2) -- (1, 0, -1);
    \draw [->, blue, thin] (1, 0, -1) -- (1, 0, 0);
    
    \draw [->, black, thin] (1, 0, 0) -- (2,1,0);
    \draw [->, black, thin] (2,1, 0) -- (2,2,0);
    \draw [->, black, thin] (2,2, 0) -- (3,3,0);
    \draw [->, black, thin] (3,3, 0) -- (4,4,0);

    \draw [->, blue, very thick] (3, 0, -2) -- (2, 0, -1);
    \draw [->, black, very thick] (2, 0, -1) -- (2, 1, -1);
    \draw [->, black, very thick] (2, 1, -1) -- (2, 2, -1);
    \draw [->, black, very thick] (2, 2, -1) -- (2, 3, -1);
    \draw [->, black, very thick] (2, 3, -1) -- (2, 4, -1);

    \end{tikzpicture}
}

\newcommand{\pictureInvolutionSlide}{
\begin{tikzpicture}[x=1.2cm,y=1cm,z=0.5cm,rotate around y=33,rotate around x=0]\label{slide1}

    \pictureInvolutionSetup

    \drawrectZ{orange, very thick}{0}{-1}{4}{0}{4};
    \drawrectZ{gray, thick}{-1}{-1}{4}{0}{4}
    \draw [->, blue, very thick] (1, 0, -3) -- (0, 0, -2);
    \draw [->, blue, very thick] (0, 0, -2) -- (0, 0, -1);
    \draw [->, blue, very thick] (0, 0, -1) -- (0, 0, 0);
    \draw [->, blue, dashed, thin] (0, 0, -1) -- (-1, 0, 0);
    
    \draw [->, black,  thick] (-0.02, 0, 0) -- (0.98, 1, 0);
    \draw [->, black,  thick] (0.95, 1, 0) -- (1, 2, 0);
    \draw [->, black, dashed,  thin] (-1, 0, 0) -- (0, 1, 0);
    \draw [->, black, dashed,  thin] (0, 1, 0) -- (1, 2, 0);
    \draw [->, black, thick] (1, 2, 0) -- (2, 3, 0);
    \draw [->, black, thick] (2, 3, 0) -- (3, 4, 0);
    
    \draw [->, blue, very thick] (2, 0, -3) -- (2, 0, -2);
    \draw [->, blue, very thick] (2, 0, -2) -- (1, 0, -1);
    \draw [->, blue, very thick] (1, 0, -1) -- (0,0,0);
    \draw [->, blue, dashed, thin] (1, 0, -1) -- (1, 0, 0);
    \draw [->, blue, dashed, thin] (1, 0, -1) -- (1, 0, 0);
    
    \draw [->, black,  thick] (0.02, 0, 0) -- (1.02, 1, 0);
    \draw [->, black,  thick] (1.05, 1, 0) -- (2, 2, 0);
    \draw [->, black, dashed, thin] (1, 0, 0) -- (2,1,0);
    \draw [->, black, dashed, thin] (2,1, 0) -- (2,2,0);
    \draw [->, black, thick] (2,2, 0) -- (3,3,0);
    \draw [->, black, thick] (3,3, 0) -- (4,4,0);

    \end{tikzpicture}
}

\newcommand{\pictureInvolutionProject}{
\begin{tikzpicture}[x=1.2cm,y=1cm,z=0.5cm,rotate around y=33,rotate around x=0]\label{slide1}
    \pictureInvolutionSetup
	\drawrectZ{gray, very thin}{0}{-1}{4}{0}{4};
	\drawrectZ{orange, very thick}{-1}{-1}{4}{0}{4}
	
    \draw [->, blue, very thick] (1, 0, -3) -- (0, 0, -2);
    \draw [->, blue, very thick] (0, 0, -2) -- (0, 0, -1);
    \draw [->, blue, thin] (0, 0, -1) -- (0, 0, 0);
    
    \draw [->, black,  thick] (0, 0, -1) -- (1, 1, -1);
    \draw [->, black,  thick] (1, 1, -1) -- (1, 2, -1);
    \draw [->, black, thick] (1, 2, -1) -- (2, 3, -1);
    \draw [->, black, thick] (2, 3, -1) -- (3, 4, -1);
    
    \draw[->, orange, thick] (3,4,-1) -- (3,4,0);
    
    \draw [->, blue, very thick] (2, 0, -3) -- (2, 0, -2);
    \draw [->, blue, very thick] (2, 0, -2) -- (1, 0, -1);
    \draw [->, blue, thin] (1, 0, -1) -- (0,0,0);
    
    \draw [->, black,  thick] (1, 0, -1) -- (2, 1, -1);
    \draw [->, black,  thick] (2, 1, -1) -- (3, 2, -1);
    \draw [->, black, thick] (3,2, -1) -- (4,3,-1);
    \draw [->, black, thick] (4,3, -1) -- (4,4,-1);
    \draw [->, orange, thick] (4,3, -1) -- (4,4,0);

    \draw [->, blue, very thick] (3, 0, -2) -- (2, 0, -1);
    \draw [->, black, very thick] (2, 0, -1) -- (2, 1, -1);
    \draw [->, black, very thick] (2, 1, -1) -- (2, 2, -1);
    \draw [->, black, very thick] (2, 2, -1) -- (2, 3, -1);
    \draw [->, black, very thick] (2, 3, -1) -- (2, 4, -1);
    
    \filldraw[orange] (2,1,-1) circle (2pt) node[right] {$C$};
    \end{tikzpicture}
}

\newcommand{\pictureInvolutionTranspose}{
\begin{tikzpicture}[x=1.2cm,y=1cm,z=0.5cm,rotate around y=33,rotate around x=0]\label{slide1}
    \pictureInvolutionSetup
	\drawrectZ{gray, very thin}{0}{-1}{4}{0}{4};
	\drawrectZ{orange, very thick}{-1}{-1}{4}{0}{4}
	    
    \draw [->, blue, very thick] (2, 0, -3) -- (2, 0, -2);
    \draw [->, blue, very thick] (2, 0, -2) -- (1, 0, -1);

	\draw [->, black,  very thick] (1, 0, -1) -- (2, 1, -1);
    \draw [->, black, very thick] (2, 1, -1) -- (2, 2, -1);
    \draw [->, black, very thick] (2, 2, -1) -- (2, 3, -1);
    \draw [->, black, very thick] (2, 3, -1) -- (2, 4, -1);
    
    \draw [-, blue, thin, dashed] (1, 0, -1) -- (0,0,0);
    \draw [-, black, thin, dashed] (0, 0, 0) -- (1, 1, 0);

    \draw [->, blue, very thick] (3, 0, -2) -- (2, 0, -1);
    \draw [->, blue, thin] (2, 0, -1) -- (1,0,0);

    \draw [->, black, thin] (1, 0, 0) -- (1.05, 1, 0);
    \draw [->, black,  thin] (1.05, 1, 0) -- (2, 2, 0);
    \draw [->, black, thin] (2,2, 0) -- (3,3,0);
    \draw [->, black, thin] (3,3, 0) -- (4,4,0);
    \draw [->, black, thin] (3,3, 0) -- (4,4,0);
    
    \draw [-, black, dashed,  very thick] (2, 0, -1) -- (2, 1, -1);
    \draw [-, black, dashed, very thick] (2, 1, -1) -- (3, 2, -1);
    \draw [-, black, dashed, very thick] (3,2, -1) -- (4,3,-1);
    \draw [-, black, dashed, very thick] (4,3, -1) -- (4,4,-1);

    \end{tikzpicture}
}

\newcommand{\pictureInvolutionRevslide}{
\begin{tikzpicture}[x=1.2cm,y=1cm,z=0.5cm,rotate around y=33,rotate around x=0]\label{slide1}
    \pictureInvolutionSetup
	\drawrectZ{orange, very thick}{0}{-1}{4}{0}{4};
	\drawrectZ{gray, very thin}{-1}{-1}{4}{0}{4}
	
    \draw [->, blue, very thick] (1, 0, -3) -- (0, 0, -2);
    \draw [->, blue, very thick] (0, 0, -2) -- (0, 0, -1);
    \draw [->, blue, dashed, thick] (0, 0, -1) -- (0, 0, 0);
    \draw [->, black, dashed, thick] (0, 0, 0) -- (0.95, 1, 0);
    \draw [->, black, dashed, thick] (0.95, 1, 0) -- (1, 2, 0);
    \draw [->, black, very thick] (1, 2, 0) -- (2, 3, 0);
    \draw [->, black, very thick] (2, 3, 0) -- (3, 4, 0);
    
    \draw[->,blue, very thick] (0, 0, -1) -- (-1, 0, 0);
    \draw[->,black, very thick] (-1, 0, 0) -- (0, 1, 0);
    \draw[->,black, very thick] (0, 1, 0) -- (1, 2, 0);

    \draw [->, blue, very thick] (3, 0, -2) -- (2, 0, -1);
    \draw [->, blue, thick, dashed] (2, 0, -1) -- (1,0,0);

    \draw [->, black, thick, dashed] (1, 0, 0) -- (1.05, 1, 0);
    \draw [->, black, thick, dashed] (1.05, 1, 0) -- (2, 2, 0);
    \draw [->, black, very thick] (2,2, 0) -- (3,3,0);
    \draw [->, black, very thick] (3,3, 0) -- (4,4,0);
    \draw [->, black, very thick] (3,3, 0) -- (4,4,0);

    \draw[->,blue, very thick] (2, 0, -1) -- (2, 0, 0);
    \draw[->,black, very thick] (2, 0, 0) -- (2, 1, 0);
    \draw[->,black, very thick] (2, 1, 0) -- (2, 2, 0);
    
    \end{tikzpicture}
}

\newcommand{\pictureInvolutionResult}{
\begin{tikzpicture}[x=1.2cm,y=1cm,z=0.5cm,rotate around y=33,rotate around x=0]\label{slide1}
    \pictureInvolutionSetup
	\drawrectZ{gray, very thin}{0}{-1}{4}{0}{4};
	\drawrectZ{gray, thick}{-1}{-1}{4}{0}{4}
	
    \draw [->, blue, very thick] (1, 0, -3) -- (0, 0, -2);
    \draw [->, blue, very thick] (0, 0, -2) -- (0, 0, -1);
    \draw[->,blue, thin] (0, 0, -1) -- (-1, 0, 0);
    \draw[->,black, thin] (-1, 0, 0) -- (0, 1, 0);
    \draw[->,black, thin] (0, 1, 0) -- (1, 2, 0);   
    \draw [->, black, thin] (1, 2, 0) -- (2, 3, 0);
    \draw [->, black, thin] (2, 3, 0) -- (3, 4, 0);
    
    \draw [->, blue, very thick] (2, 0, -3) -- (2, 0, -2);
    \draw [->, blue, very thick] (2, 0, -2) -- (1, 0, -1);

	\draw [->, black,  very thick] (1, 0, -1) -- (2, 1, -1);
    \draw [->, black, very thick] (2, 1, -1) -- (2, 2, -1);
    \draw [->, black, very thick] (2, 2, -1) -- (2, 3, -1);
    \draw [->, black, very thick] (2, 3, -1) -- (2, 4, -1);

    \draw [->, blue, very thick] (3, 0, -2) -- (2, 0, -1);
    \draw[->,blue, thick] (2, 0, -1) -- (2, 0, 0);
    \draw[->,black, thick] (2, 0, 0) -- (2, 1, 0);
    \draw[->,black, thick] (2, 1, 0) -- (2, 2, 0);
    \draw [->, black, thin] (2,2, 0) -- (3,3,0);
    \draw [->, black, thin] (3,3, 0) -- (4,4,0);
    \draw [->, black, thin] (3,3, 0) -- (4,4,0);

    \end{tikzpicture}
}

\title[Determinantal formulas for dual Grothendieck polynomials]{
Determinantal formulas for dual Grothendieck polynomials  
} 

\author[Alimzhan Amanov \and Damir Yeliussizov]{Alimzhan Amanov \and Damir Yeliussizov}

\address{Kazakh-British Technical University, Almaty, Kazakhstan}
\email{\href{mailto:alimzhan.amanov@gmail.com}{alimzhan.amanov@gmail.com}, \href{mailto:yeldamir@gmail.com}{yeldamir@gmail.com}}

\begin{document}

\begin{abstract}
We prove Jacobi--Trudi-type determinantal formulas for skew dual 
Grothendieck polynomials which are $K$-theoretic deformations of Schur polynomials. We also obtain a bialternant-type formula analogous to the classical definition of Schur polynomials.
\end{abstract}

\maketitle


\section{Introduction}
The {\it dual} (stable) {\it Grothendieck polynomials} 
are certain $K$-theoretic inhomogeneous deformations of Schur polynomials, 
introduced by Lam and Pylyavskyy in \cite{lp} (with earlier implicit relations in \cite{lenart, buch}, as a dual basis to stable Grothendieck polynomials \cite{fk}).  
Besides interesting combinatorial properties \cite{lp, ggl, dy, dy2, dy5}, these functions have connections with some natural probabilistic models \cite{dy3, dy6, dy4}, and hence they arise in contexts beyond $K$-theoretic Schubert calculus. 


While various aspects of these 
polynomials are developed well (e.g. in the  references above), 
their theory can still be viewed as 
somewhat fragmented. 
One gap
concerned problems about determinantal formulas (for skew shapes), which are of fundamental importance 
due to 
connections mentioned above (such formulas are key for computational reasons and as a tool for further analysis). 


In this paper, we prove several determinantal identities for the dual Grothendieck polynomials $g_{\lambda/\mu}$ (see def.~\ref{defg}). Let us state one of our main results.

\begin{theorem}[Dual Jacobi--Trudi-type formula]\label{main}
Let $n \ge \lambda_1, \mu_1$. The following determinantal identity holds 
\begin{align}\label{one}
g_{\lambda/\mu}(x_1, x_2,\ldots) 
= \det\left[e_{\lambda'_i - i -\mu'_j + j}(1^{\lambda'_i-1 - \mu'_j}, x_1, x_2, \ldots )\right]_{1 \le i, j \le n}.
\end{align}
Here 
$\{e_n\}$ are the elementary symmetric functions; $\lambda'$ is the conjugate partition of $\lambda$; and $(1^k) := (1, \ldots, 1)$ repeated $k$ times which is defined as $\varnothing$ for $k \le 0$.
\end{theorem}

This formula is of dual Jacobi--Trudi type or an analogue of the N\"agelsbach--Kostka formula 
for Schur functions, see e.g. \cite{macdonald, sta}. 
In a more general refined form, which we also prove in Sec.~\ref{proofmain}, this identity was conjectured by Grinberg \cite{grinberg}. In the straight shape case, 
Jacobi--Trudi-type formulas for dual Grothendieck polynomials 
were first given in \cite{sz}; the identity \eqref{one} for $g_{\lambda}$ was proved in \cite{ln} via specializations of Schubert polynomials; in refined form this identity was proved combinatorially in \cite{dy}.  


We also prove a dual formula (Theorem~\ref{dual}) written via the complete homogeneous symmetric functions $\{h_n\}$, and derive a bialternant-type formula (Theorem~\ref{bialt}) which is analogous to the classical definition of Schur polynomials. 


Combinatorial proofs of such identities always rely on the Lindstr\"om--Gessel--Viennot (LGV) lemma \cite{lgv} via lattice paths. But perhaps the most surprising aspect of our proof of Theorem~\ref{main} is a 3-dimensional lattice construction. 
To establish positivity, after the LGV lemma this system needs a non-local sign-reversing involution which is the most difficult part of the proof. 
Even though the operations which we describe look somewhat technical, they are essentially like {\it jeu de taquin} type operations on tableaux.  
See Sec.~\ref{proofmain} for details. 

We also note that a different and independent proof of this result was given by J.~S.~Kim \cite{jskim} almost the same time as our paper was written. An independent proof of the dual formula in Theorem~\ref{dual} appeared in S.~Iwao's work \cite{iwao} using fermionic techniques. 

\section{Preliminaries}
A {\it partition} is a sequence $\lambda = (\lambda_{1}, \ldots, \lambda_{\ell})$ of positive integers $\lambda_1 \ge \cdots \ge \lambda_{\ell}$, where $\ell(\lambda) = \ell$ is the {\it length} of $\lambda$ (we set $\lambda_i = 0$ for $i > \ell$). Every partition $\lambda$ can be represented as the {\it Young diagram} $\{(i,j)  : i \in [1, \ell], j\in [1,\lambda_{i}]\}$. We denote by $\lambda'$ the {\it conjugate} partition of $\lambda$ whose diagram is transposed. We draw diagrams in English notation. For $\lambda \supset \mu$, the {\it skew shape} $\lambda/\mu$ has the diagram of $\mu$ removed from the diagram of $\lambda$.

A {\it reverse plane partition} (RPP) of shape $\lambda/\mu$ is a filling of the boxes of the diagram of $\lambda/\mu$ with positive integers weakly increasing along rows from left to right and along columns from top to bottom. Let $\mathrm{RPP}(\lambda/\mu)$ be the set of RPPs of shape $\lambda/\mu$.

\begin{definition}[\cite{lp}]\label{defg}
The {\it skew dual Grothendieck polynomials} $g_{\lambda/\mu}$ are defined as follows:
$$
g_{\lambda/\mu}(\mathbf{x}) := \sum_{T\, \in\, \mathrm{RPP}(\lambda/\mu)} \prod_{i \ge 1} x_i^{c_i(T)},
$$
where $c_i(T)$ is the number of columns of $T$ containing the entry $i$, and $\mathbf{x} = (x_1, x_2, \ldots)$.
\end{definition}

It is known that $g_{\lambda/\mu}$ is a symmetric function \cite{lp}, which is inhomogeneous and its top degree component is the Schur polynomial $s_{\lambda/\mu}$.

\section{Proof of Theorem~\ref{main}}\label{proofmain}
\subsection{Refined version} We will prove a more general version of Theorem~\ref{main} for {\it refined} dual Grothendieck polynomials $g_{\lambda/\mu}(\mathbf{t}; \mathbf{x})$ introduced in \cite{ggl} which contain  the extra parameters $\mathbf{t} = (t_1, t_2, \ldots)$ so that $g_{\lambda/\mu}(\mathbf{1}; \mathbf{x}) = g_{\lambda/\mu}(\mathbf{x})$ (i.e. when $t_i = 1$ for all $i$). We define these formal power series as follows
$$
g_{\lambda/\mu}(\mathbf{t}; \mathbf{x}) := \sum_{T \in \mathrm{RPP}(\lambda/\mu)} \prod_{i \ge 1} x_i^{c_i(T)} \prod_{j \ge 1} t_j^{d_j(T)},
$$
where $d_j(T)$ is the number of entries in $j$-th row of $T$ that are equal to an entry directly below. The following generalization of Theorem~\ref{main} 
was conjectured by Grinberg \cite{grinberg}. 
\begin{theorem}
Let $n \ge \lambda_1, \mu_1$. The following determinantal identity holds 
\begin{align}\label{two}
g_{\lambda/\mu}(\mathbf{t}; \mathbf{x}) = \det\left[e_{\lambda'_i - i - \mu'_j + j}(t_{\mu'_j + 1}, \ldots, t_{\lambda'_i - 1}, \mathbf{x}) \right]_{1 \le i,j \le n}.
\end{align}
\end{theorem}

\subsection{Proof overview} First, we present a 3d lattice path system whose path enumerators give the right-hand side of \eqref{two}. Then the most difficult part of the proof contains the description of a sign-reversing weight-preserving involution on such path systems which cancels out negative terms. To describe it, we introduce several operations on path systems (subsec.~\ref{invo}). Finally, we show a correspondence between {\it good} path systems (fixed by the involution) and skew RPPs, which gives the formula for refined dual Grothendieck polynomials.

\subsection{A 3d lattice path system}
Let $m$ be the number of variables in $\mathbf{x}$ (assumed finite for now, we can afterwards let $m \to \infty$). 
Let $\lambda$, $\mu$ and $n \ge \lambda_1, \mu_1$ be all fixed. Consider the lattice $\mathbb{Z}^3$ with the following weighted edges $e$ (which we call \textit{steps}):
\begin{itemize}
    \item on the planes $z = k$ (the {\it walls}):
    \begin{itemize}
        \item[~] $(i,j,k) \to (i,j+1,k)$ with the weight $w(e) = 1$;
        \item[~] $(i,j,k) \to (i+1, j+1, k)$ with the weight $w(e) = x_{m - j}$ for $j \in [0,m-1]$;
    \end{itemize}
    
    \item on the plane $y = 0$ (the {\it floor}):
    \begin{itemize}
        \item[~] $(i,0,j) \to (i,0,j-1)$ with the weight $w(e) =t_j$ for $j > 0$;
        \item[~] $(i,0,j) \to (i-1,0,j-1)$ with the weight $w(e) = 1$.
    \end{itemize}
\end{itemize}
(All other weights are $0$.) See Fig.~\ref{fig:graphexample}. Note that
any (nonzero weight) path which starts on the floor plane and ends on a wall plane, begins with steps on the floor plane (each decrementing the $z$-coordinate) and then continues with steps on the wall plane (incrementing the $y$-coordinate). 
\begin{figure}
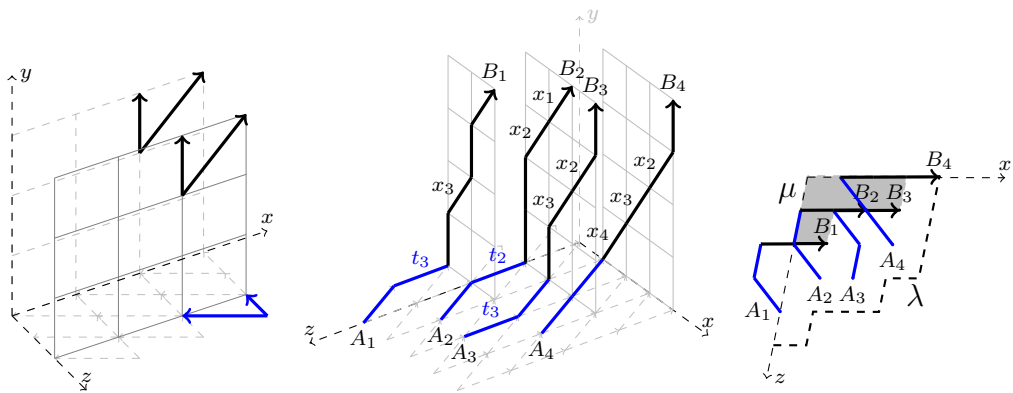

    \centering
    {\scriptsize
    \picturegraph
    \picturenpath
    \picturenpathy    
     }
    \caption{(a) A part of the lattice with typical sample steps. (b) An example of a path system for $\lambda'/\mu' = (5443)/(211)$, $n = 4$ and $m = 4$ with the path weights given by $w(P_1) = t_3 x_1 x_3,$ $w(P_2) = t_2 x_1 x_2,$ $w(P_3) = t_3 x_2 x_3,$ $w(P_4) =  x_2 x_3 x_4$. (c) A view of (b) from above with $\lambda$ and $\mu$ shown.}
    \label{fig:graphexample}
\end{figure}

Let us define $n$ {\it source} points $\mathbf{A} = (A_1, \ldots, A_n)$ and $n$ {\it sinks}  $\mathbf{B} = (B_1, \ldots, B_n)$ whose coordinates are given by $A_i = (i-1, 0, \lambda'_i - 1)$ and $B_i = (i,m,\mu'_i)$ for all $i \in [1,n]$. Define weighted path enumerators in the usual way
$$w(A \to B) := \sum_{P: A \to B} \prod_{e \in P} w(e)$$
over all paths $P$ in the lattice from the point $A$ to the point $B$ with steps $e$ given as above. The following formula is then clear from the construction.
\begin{lemma}
We have
\begin{align*} 
w(A_i \to B_j) = e_{\lambda'_i - i - \mu'_j + j}(t_{\mu'_j + 1}, \ldots, t_{\lambda'_i - 1}, x_1, \ldots, x_m).
\end{align*}
\end{lemma}
Similarly, define the signed weighted multi-enumerators 
\begin{align}\label{wdef}
w(\mathbf{A} \to \mathbf{B}) := \sum_{ \mathbf{P} \in {N}(\mathbf{A}, \mathbf{B})} \sgn(\mathbf{P})\, w(\mathbf{P}),
\end{align}
where ${N}(\mathbf{A}, \mathbf{B})$ is the set of nonintersecting path systems $\mathbf{P}$ from $\mathbf{A}$ to $\mathbf{B}$ (i.e. $n$ paths with no common vertices as in Fig.~\ref{fig:graphexample}(b)), $\sgn(\mathbf{P}) := \sgn(\sigma)$ for $\sigma \in S_n$ if $\mathbf{P}$ joins $A_i$ with $B_{\sigma(i)}$ for all $i \in [1,n],$ and $w(\mathbf{P})$ is the product of weights of all edges in the whole system $\mathbf{P}$. 

Since our directed graph 
is acyclic, by the LGV lemma we immediately obtain the following.
\begin{corollary}
We have
\begin{align*}
\det\left[e_{\lambda'_i - i - \mu'_j + j}(t_{\mu'_j + 1}, \ldots, t_{\lambda'_i - 1}, x_1, \ldots, x_m) \right]_{1 \le i,j \le n} = w(\mathbf{A} \to \mathbf{B}).  
\end{align*}
\end{corollary}
Note that the right-hand side may contain terms with negative signs since we have a 3d lattice and there are many nonintersecting path systems corresponding to non-identity permutations.

\subsection{Path transformations and sign-reversing involution}\label{invo} In this subsection we describe certain path operations and a sign-reversing weight-preserving involution on $N(\mathbf{A}, \mathbf{B})$ which leaves only positive terms in 
\eqref{wdef}. 
Suppose we have a nonintersecting path system $\mathbf{P} = (P_1, \ldots, P_n)$ from $\mathbf{A}$ to $\mathbf{B}$ where the paths are ordered with respect to the sinks ordering, i.e. $P_i$ has sink $B_i$. We are going to introduce several important definitions necessary for path transformations.

\begin{definition}[Projections and intersections]\label{defproj} Let us define the following notions:

(Path projections to planes) For a path $P_i$ from (some) $A_j$ to $B_i$, the {\it projection} of $P_i$ on the plane $z = k$ with $k \ge \mu'_i$ is defined as follows: take the first intersection point $C$ of $P_i$ with the plane $z = k$ and copy the part of $P_i$ from the plane $z = \mu'_i$ to the plane $z = k$ starting from $C$ 
and bounded by the plane $x = i$ so that when it hits that plane, it continues on it (by increasing the $y$-coordinate) to the point $(i,m,k)$. 
See Fig.~\ref{fig:projection}~(a). 
(If no such point $C$ exists or if its $x$-coordinate is larger than $i$, then the projection is undefined.
Note also that if the projection is defined, then $k \le \lambda'_j - 1$.)

(Projection intersections) We say that two paths 
{\it intersect on the plane} $z = k$ if their projections on the plane $z = k$ are defined (as above) and intersect.

(Path system intersections) A path system has {\it no intersections on the plane} $z = k$ if no pair of paths in this system intersects on the plane $z = k$.
\end{definition}
\begin{definition}[Cut edges] 
When we project $P_i$ on $z = k$, 
some (wall) part of $P_i$ disappears on the projection. 
To keep track of such edges and (floor) edges of $P_i$ lying between the planes $z = \mu'_i$ and $z = k$, we define \textit{cut edges} lying on the plane $x = i$. 
Let $T_{i,k}$ be the first intersection point of the plane $x = i$ and the projection of $P_i$ on the plane $z = k$ (whenever projection is defined). See again Fig.~\ref{fig:projection}~(a).
Note that for fixed $i$, the points $T_{i,k}$ have non-increasing $y$-coordinates. 
Let $f_k$ be the (single floor) edge of $P_i$ lying between the planes $z = k$ and $z = k + 1$. 
We then define \textit{$k$-th cut edge} of two types:
    
    (1) If $T_{i,k}$ and $T_{i,k+1}$ have the same $y$-coordinate, then we must have $w(f_k) = t_{k+1}$ and the segment $T_{i,k}\, T_{i,k+1}$ is the $k$-th cut edge of weight $t_{k+1}$;
    
    (2) Otherwise, $T_{i,k+1}$ is below $T_{i,k}$ (by $y$-coordinate), $w(f_k) = 1$, and there is a unique (wall) 
    edge $e_k$ of $P_i$ of weight $w(e_k) = x_s$ for some $s$, such that this edge is seen in the projection of $P_i$ on $z = k$, but disappears in the projection on $z = k + 1$. 
    Let $T'_{i,k+1} = T_{i,k} + (0,-1,1)$. 
    Then the segment $T_{i,k}\, T'_{i,k+1}$ is the $k$-th cut edge of weight $x_s$.



For $k$-th cut edges $e$ and $e'$ (corresponding to different $i$), $e$ is {\it larger} than $e'$ if the endpoint of $e$ on the plane $z = k$ is higher than the endpoint of $e'$ on the same plane. 
\end{definition}
\begin{figure}
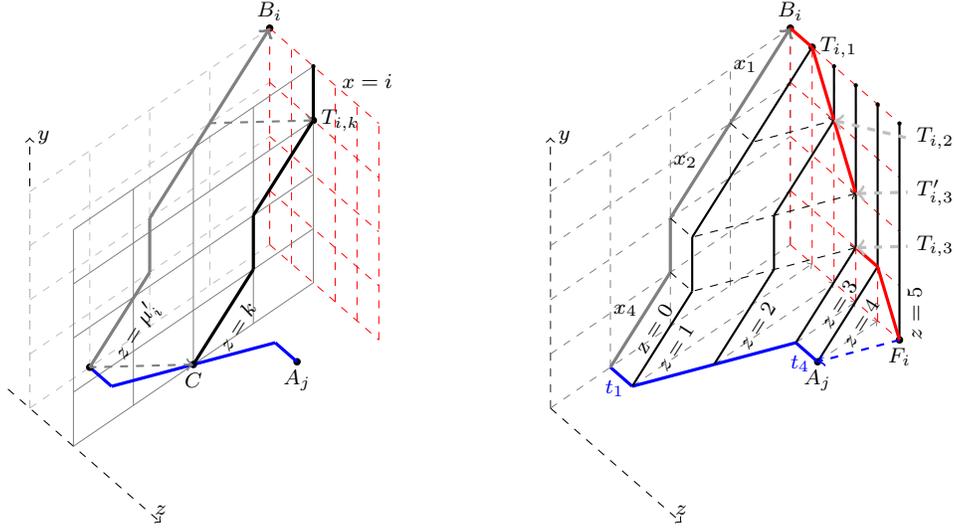

    \centering
    {\scriptsize
    \pictureprojj
    \qquad\qquad\qquad \pictureproj
    }
    \caption{(a) Projection of a path $P_i : A_j \to B_i$ on the plane $z = k$. (b) Several projections (black solid) and cut edges (red solid). 
    }
    \label{fig:projection}
\end{figure}
\begin{example}
Consider Fig.~\ref{fig:projection}~(b) where cut edges are shown as red solid segments. For $k = 0$ and the planes $z = 0$ and $z = 1$ we have $B_i = T_{i,0} = (4,4,0)$, $T_{i,1} = (4,4,1)$, $w(f_0) = t_1$. Then the $0$-th cut edge is the segment $T_{i,0}\, T_{i,1}$ of weight $t_1$ (type 1).
For $k = 2$ and the planes $z = 2$ and $z = 3$, we have $w(f_2) = 1$, $T_{i,2} = (4,3,2)$, $T'_{i,3} = (4,2,3)$, $T_{i,3} = (4,1,3)$ and $w(e_{i+1}) = x_2$. Then the $2$-nd cut edge is the segment $T_{i,2}\, T'_{i,3}$ of weight $x_2$ (type 2).
\end{example}


\begin{definition}[Step and slide operations]\label{defstep}
The operation $\mathsf{step}_k$ on $\mathbf{P}$ is defined as follows:
    
    (1) Choose minimal index $i$ such that $P_i$ and $P_{i+1}$ intersect on the plane $z = k + 1$, but do not intersect on the plane $z = k$; if there is no such index, do nothing;
    
    (2) Let $e_i$ and $e_{i+1}$ be the edges of $P_i$ and $P_{i+1}$ between the planes $z = k$ and $z = k + 1$. By the choice of $i$, 
    we must have the weights $w(e_{i}) = 1$ and $w(e_{i+1}) = t_{k+1}$ (other cases do not produce intersections on $z = k + 1$).
    Let $C$ be the first common point of projections of $P_i$ and $P_{i+1}$ on the plane $z = k + 1$. Let $e'_{i}$ and $e'_{i+1}$ be the edges of projections of $P_i$ and $P_{i+1}$ on $z = k + 1$ preceding $C$ (i.e. $C$ is an endpoint of these edges). We must have $w(e'_{i}) = x_s$ and $w(e'_{i+1})=1$. 
    We then modify these paths as follows. In $P_i$ we switch the edge $e_{i}$ (of weight $1$) to the edge of weight $t_{k+1}$ with same starting point, shift the part of $P_i$ between the edges $e_{i}$ and $e'_{i}$ by $1$ to the right (towards the $x$-axis), and switch the edge $e'_{i}$ (of weight $x_s$) to the edge of weight $1$ with the same ending point (other parts of $P_i$ remain intact). The path $P_{i+1}$ is modified in the same way; we `exchange' edges $e_{i}$ with $e_{i+1}$, and `exchange' $e'_{i}$ with $e'_{i+1}$.
    See Fig.~\ref{fig:stepk}.

Define now the operation $\mathsf{slide}_k$ on $\mathbf{P}$
by repeatedly performing the operation $\mathsf{step}_k$ until it makes no changes in the path system. See Fig.~\ref{fig:slidek}.
\begin{remark}
One can observe that intersecting paths on each $\mathsf{step}_k$ are actually touching. 
\end{remark}

\end{definition}
   \begin{figure}
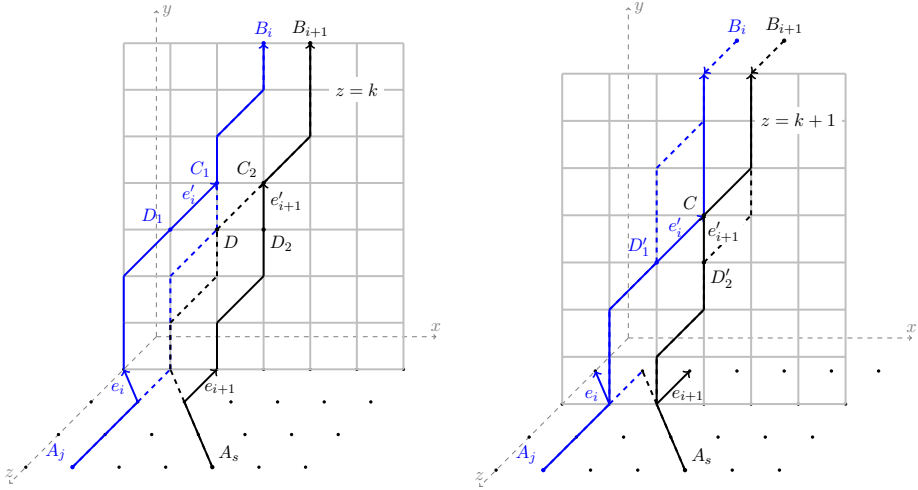

        \begin{center}
            \centering
            \begin{minipage}{0.4\textwidth}
                \centering
                \resizebox{\textwidth}{!}{
                    \picturestepk
                }
            \end{minipage}
            \begin{minipage}{0.4\textwidth}
                \centering
                \resizebox{\textwidth}{!}{%
                    \picturestepkinv%
            }
            \end{minipage}
        \caption{The paths $P_i$ and $P_{i+1}$ projected to the planes $z = k$ (left) and $z = k + 1$ (right). The dashed fragments show how the operation $\mathsf{step}_k$ will modify the paths. 
        }
        \label{fig:stepk}
        \end{center}
    \end{figure}
    

\begin{definition}[Inverse steps and slide operations]
The operation $\mathsf{rstep}_k$ is defined as follows:
    
    (1) Choose maximal index $i$ such that $P_i$ and $P_{i+1}$ intersect on the plane $z = k$, but do not intersect on the plane $z = k + 1$, and $k$-th cut edge of $P_i$ is not below $k$-th cut edge of $P_{i+1}$. If there is no such index, do nothing.
    
    (2) Let $e_{i}$ and $e_{i+1}$ be edges of $P_i$ and $P_{i+1}$ between the planes $z = k$ and $z = k + 1$. Again, by the choice of $i$ we must have the weights $w(e_{i}) = t_{k+1}$ and $w(e_{i+1}) = 1$.  Let $D$ be the last common point of projections of $P_i$ and $P_{i+1}$ on $z = k$. Let $e'_{i}$ and $e'_{i+1}$ be the edges of projections of $P_i$ and $P_{i+1}$ to the plane $z = k$ succeeding $D$ (i.e. $D$ is an endpoint of these edges). We must have $w(e'_{i}) = 1$ and $w(e'_{i+1}) \neq 1$. We then `exchange'  $e_{i}$ with $e_{i+1}$, and `exchange' $e'_{i}$ with $e'_{i+1}$, similarly as in the previous definition.

Define the operation $\mathsf{rslide}_k$ on path systems that do not intersect on the plane $z = k + 1$,
by repeatedly performing the operation $\mathsf{rstep}_k$ until it makes no changes in the path system. See Fig.~\ref{fig:slidek} in the reverse direction.
\end{definition}


\begin{figure}
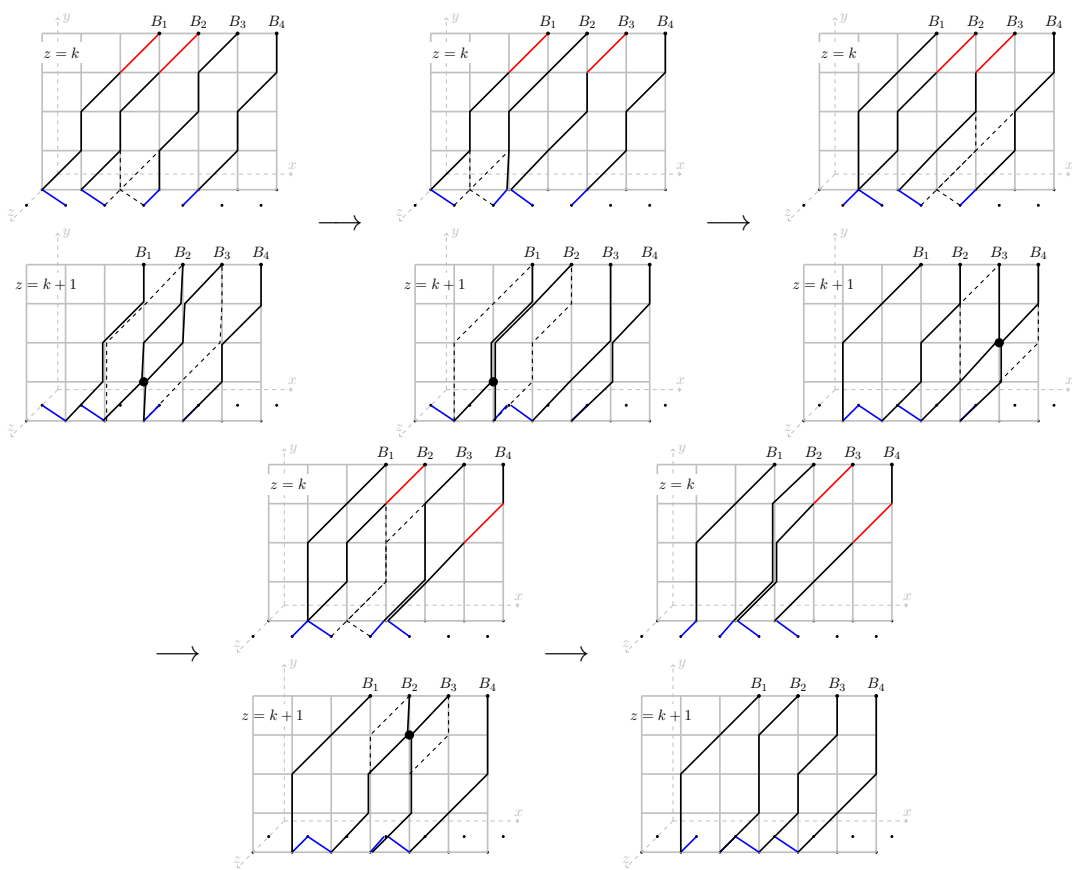

    \begin{center}
        \centering
        \begin{minipage}{0.27\textwidth}
            \centering
            \resizebox{\textwidth}{!}{
                \picturewavesA
            }
            \resizebox{\textwidth}{!}{
                \picturewavesAinv
            }
        \end{minipage}
        $\longrightarrow{}$
        \begin{minipage}{0.27\textwidth}
            \centering
            \resizebox{\textwidth}{!}{
                \picturewavesB
            }
            \resizebox{\textwidth}{!}{
                \picturewavesBinv
            }
        \end{minipage}
        $\longrightarrow{}$
        \begin{minipage}{0.27\textwidth}
            \centering
            \resizebox{\textwidth}{!}{
                \picturewavesC
            }
            \resizebox{\textwidth}{!}{
                \picturewavesCinv
            }
        \end{minipage}
        
        $\longrightarrow{}$
        \begin{minipage}{0.27\textwidth}
            \centering
            \resizebox{\textwidth}{!}{
                \picturewavesD
            }
            \resizebox{\textwidth}{!}{
                \picturewavesDinv
            }
        \end{minipage}
        $\longrightarrow{}$
        \begin{minipage}{0.27\textwidth}
            \centering
            \resizebox{\textwidth}{!}{
                \picturewavesE
            }
            \resizebox{\textwidth}{!}{
                \picturewavesEinv
            }
        \end{minipage}
    \caption{An example of $\mathsf{slide}_k$ operation where $\mathsf{step}_k$ is applied four times. Each step shows path system projections on the planes $z = k, k+1$ (top, bottom). The red edges disappear on $z = k + 1$. The dashed lines show how $\mathsf{step}_k$ is applied. When  $\mathsf{rslide}_k$ is applied to the resulting system, it operates in the reverse direction and results in the initial path system.}   
    \label{fig:slidek}
    \end{center}
\end{figure}

\begin{lemma}\label{slide}
The following properties hold:
    
    (i) Suppose $\mathbf{P}$ has no intersections on the plane $z = k$. Then after performing $\mathsf{slide}_k$, the paths with sinks on the planes $z = 0,\ldots,k$ have no intersections on the plane $z = k + 1$.
    
    (ii) Suppose $\mathbf{P}$ has no intersections on the plane $z = k + 1$ and $k$-th cut edges are non-increasing (from left to right). Then after performing $\mathsf{rslide}_k$, the resulting system has no intersections on the plane $z = k$.
    
    (iii) Suppose $\mathbf{P}$ has no intersections on the plane $z = k$. Then after performing $\mathsf{slide}_k$ and then $\mathsf{rslide}_k$ we get back to $\mathbf{P}$. 
\end{lemma}
\begin{proof}
It is illustrative to check the statements on Fig.~\ref{fig:slidek}.
(i) First, let us show the that the following is true on each   $\mathsf{step}_k$: it takes two leftmost paths $P_i$ and $P_{i+1}$ which do not intersect on $z = k$ but intersect on $z = k + 1$, and turns them into paths that do not intersect on the plane $z = k + 1$ but intersect on $z = k$. Recall the intersection point $C$ from Def.~\ref{defstep}. Denote the projections of $P_{i}$ and $P_{i+1}$ on the plane $z = j$ ($j = k,k+1$) as $P^{(j)}_{i}$ and $P^{(j)}_{i+1}$. Let $C_1$ and $C_2$ be the points of $P^{(k)}_{i}$ and $P^{(k)}_{i+1}$ that project to $C$ on $z = k + 1$. Additionally, denote the point $D_1$ as a preceding point of $C_1$ in $P^{(k)}_{i}$ and the point $D_2$ preceding $C_2$ in $P^{(k)}_{i+1}$, see Fig.~\ref{fig:stepk}. Observe that after the operation, the parts of $P^{(k)}_{i}$ and $P^{(k)}_{i+1}$ starting from the points $C_1$ and $C_2$ respectively, remain unchanged. Moreover, $D_1$ and $D_2$ will coincide, let us call this point by $D$, which will be the last common point of the paths $P^{(k)}_i$ and $P^{(k)}_{i+1}$ on $z = k$ after applying the operation. Now, let us show that $P^{(k+1)}_{i}$ and $P^{(k+1)}_{i+1}$ will not intersect after the operation $\mathsf{step}_k$. Similarly, denote preceding points of $C$ in $P^{(k+1)}_{i}$ and $P^{(k+1)}_{i+1}$ as $D'_1$ and $D'_2$, respectively. Then the part of the path $P^{(k+1)}_i$ until the point $D'_1$ will stay unchanged; the same is true for $D'_2$ and $P^{(k+1)}_{i+1}$. The parts starting from $D'_1$ and $D'_2$ will be nonintersecting as well, since they spread out in different directions.

    So the application of $\mathsf{step}_k$ makes two neighbouring paths $(P_i, P_{i+1})$ nonintersecting on $z = k + 1$ and intersecting on $z = k$. Assume that during the operation $\mathsf{slide}_k$, we applied it 
    to the pair of paths $(P_i, P_{i+1})$ for the first time, call this moment $T$. If the next pair to be processed is $(P_{i - 1}, P_i)$, then after $\mathsf{step}_k(P_{i-1},P_{i})$ we will have the pair $(P_i, P_{i+1})$ nonintersecting on both planes $z = k, k + 1$. Furthermore, if one applies $\mathsf{step}_k$ to the pair $(P_{i + 1}, P_{i+2})$ 
    for the first time after the moment $T$,
    then the pair $(P_i, P_{i+1})$
    will not intersect on $z = k$ (but possibly on $z = k + 1$). 
    Repeatedly combining this and the fact that the edges of weights $t_{k+1}$ move to the left on each $\mathsf{step}_k$, we see that the process $\mathsf{slide}_k$ is finite and each pair of paths $(P_i, P_{i + 1})$ will not intersect on $z = k + 1$ when $\mathsf{slide}_k$ completes the action. 

(ii) Since $k$-th cut edges are non-increasing from left to right, $\mathsf{rstep}_k$ will have effect symmetric to $\mathsf{step}_k$, one can observe that $\mathsf{step}_k$ and $\mathsf{rstep}_k$ are symmetric operations, if we reflect the lattice with respect to the line $x = 0$. 
Then in the above observations, the meaning of the points $D$ and $C$ will be swapped (i.e. $D$ becomes the first point of intersection on the next plane); similarly, $C_1$ and $C_2$ will be swapped with $D'_1$ and $D'_2$, and so on. Therefore, $\mathsf{slide}_k$ and $\mathsf{rslide}_{k}$ are symmetric and the statement follows as in (i).
    
(iii) It is enough to note that $\mathsf{slide}_k$ guarantees that after the operation, $k$-th cut edges will be non-increasing from left to right. Then by observing the symmetry in (ii) and the properties shown in (i), one can see that 
$\mathsf{rslide}_k$ operates in the reverse direction as $\mathsf{slide}_k$ and results in the initial path system.
\end{proof}

\begin{definition}[Transpose operation]
Define the operation $\mathsf{transpose}_k$ as follows:
    
    (1) If there is no pair of paths intersecting on the plane $z = k$ such that one of them has a sink on this plane, then do nothing.
    
    (2) Otherwise, let us consider intersection points between the paths with sinks on the plane $z = k$ and projections of other paths on $z = k$; among these intersections let $C$ be the leftmost and lowest point. Let $P_i$ and $P_j$ be paths $(i < j)$ whose projections pass through $C$ (so that $P_i$ has a sink on $z = k$).
    Assume that $P_j$ has a sink on the plane $z = k'$ for $k' \leq k$ and let $P_i: A_{\ell} \to B_i$,  $P_j: A_r \to B_j$. Let also $C'$ be the point of $P_j$ that projects to $C$ on the plane $z = k$ (preimage of $C$); $E_i$ be the first common point of $P_i$ with the plane $z = k$, and similarly define the point $E_j$ for the path $P_j$; let $F_j$ be the first common point of $P_j$ and the plane $z = k'$. The defined points split the path $P_i$ into three parts and $P_j$ into four parts. Let us change $P_i$ and $P_j$ as follows: 
    \begin{align*}
        & P_i := P_{i}(A_{\ell} \to E_i)\, P_{j}(E_j \to F'_j)\, P_{i}(E_i \to C)\, P_{j}(C' \to B_j),\\
        & P_j := P_{j}(A_r \to E_j)\, P_{j}(F_j \to C')\, P_{i}(C \to B_i).
    \end{align*}
    See Fig.~\ref{fig:transpose}.

\begin{figure}
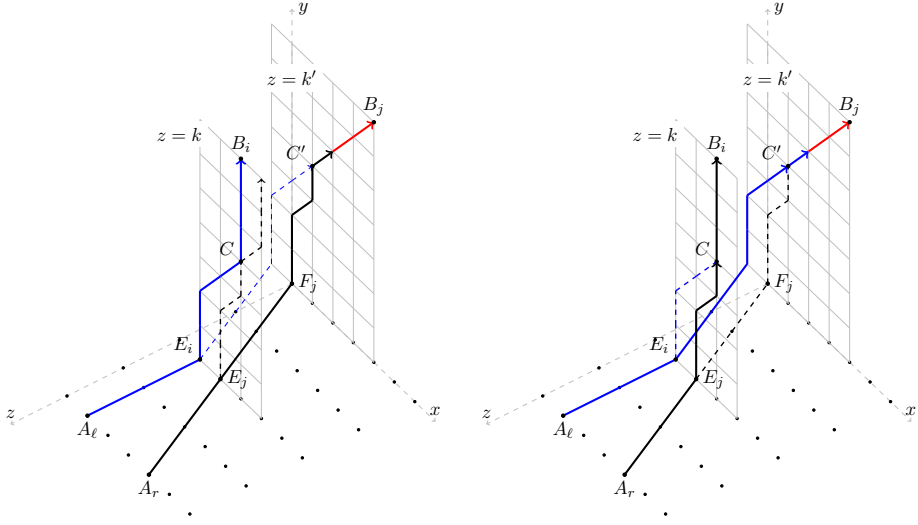

        \begin{center}
            \centering
            \begin{minipage}{0.4\textwidth}
                \centering
                \resizebox{\textwidth}{!}{
                    \picturetransBefore
                }
            \end{minipage}
            \begin{minipage}{0.4\textwidth}
                \centering
                \resizebox{\textwidth}{!}{
                    \picturetransAfter
            }
            \end{minipage}
        \caption{An example of $\mathsf{transpose}_k$: before (left) and after (right). Note that wall edges (red) corresponding to cut edges between planes $z = k'$ and $z = k$ are unaffected.}   
        \label{fig:transpose}
        \end{center}
    \end{figure}
\end{definition}

\begin{lemma}\label{transpose}
The operation $\mathsf{transpose}_k$  is an involution.  
\end{lemma}
\begin{proof}
Observe that before applying the $\mathsf{transpose}_{k}$ and after applying it, 
the path projections on the plane $z = k$ remain the same. 
\end{proof}

We are now ready to state the definition of the main involution.
\begin{definition}[The sign-reversing weight-preserving involution]\label{definvolution}
Denote 
\begin{align*}
    {s}_i & := \mathsf{slide}_i \cdot \mathsf{slide}_{i-1} \cdot \ldots \cdot \mathsf{slide}_0 \quad\text{ and }\quad 
    {s}^{-1}_i := \mathsf{rslide}_0 \cdot \mathsf{rslide}_{1} \cdot \ldots \cdot \mathsf{rslide}_i
\end{align*}
The operation ${s}_i$ expresses sequential sliding starting from the plane $z = 0$ all the way to the plane $z = i+1$, where $i \in [0,\ell(\lambda) - 1]$. Define the map $\phi : {N}(\mathbf{A}, \mathbf{B}) \to {N}(\mathbf{A}, \mathbf{B})$ as follows. Let $\mathbf{P} \in {N}(\mathbf{A}, \mathbf{B})$. 
Choose minimal index $k \in [0,\ell(\lambda)-1]$ such that ${s}^{-1}_k \cdot \mathsf{transpose}_{k+1} \cdot {s}_k (\mathbf{P}) \neq \mathbf{P}$. If there is no such $k$, then $\phi(\mathbf{P}) := \mathbf{P}$. Otherwise, set 
$$\phi(\mathbf{P}) := {s}^{-1}_k \cdot \mathsf{transpose}_{k+1} \cdot {s}_k (\mathbf{P}).$$ 
Denote by $I_\phi$ the set of fixed points of $\phi$, called {\it good} path systems. 
\end{definition}

An example  for the involution $\phi$ is shown in Fig.~\ref{fig:torpp}.

\begin{lemma}\label{involution}
We have: $\phi$ is a sign-reversing weight-preserving involution and  
\begin{align}\label{sump}
w(\mathbf{A} \to \mathbf{B}) = \sum_{\mathbf{P} \in I_\phi} w(\mathbf{P}).
\end{align}
\end{lemma}
\begin{proof}
Let $\mathbf{P} \in {N}(\mathbf{A}, \mathbf{B})$ and suppose $\mathbf{P}$ is not good.  Observe that the map $\phi$ works as follows.
    First, we project the paths of $\mathbf{P}$ whose sinks are on the planes $z = 0,\ldots,k$, to the plane $z = k + 1$ (by applying the operation $s_k$).
    By Lemma~\ref{slide} (i), the resulting projections are nonintersecting on the plane $z = k+1$.
    By the choice of $k$, the paths whose sinks are on the plane $z = k + 1$, have intersections with the projected paths (from the planes $z = 0, \ldots, k$); we then perform the transposition operation on certain pair of paths intersecting on $z = k + 1$ (by applying the operation $\mathsf{transpose}_{k+1}$).
    The operation $s_k$ guarantees non-increasing order of cut edges (from left to right), and $\mathsf{transpose}_{k+1}$ preserves this property.
    Then we perform backward slides preserving the absence of path intersections (by applying the operation $s^{-1}_{k}$).
    By Lemma~\ref{slide} (ii), the backward slides guarantee that the paths will not intersect. 
    Now, by Lemma~\ref{slide} (iii) and Lemma~\ref{transpose} it is then clear that $\phi$ is  a  sign-reversing weight-preserving involution.
Notice also that good path systems correspond to the identity permutation and hence have positive sign (otherwise the paths creating an inversion must have intersection on at least one of the planes containing the sinks of these paths). Hence \eqref{sump} follows.
\end{proof}
\begin{figure}
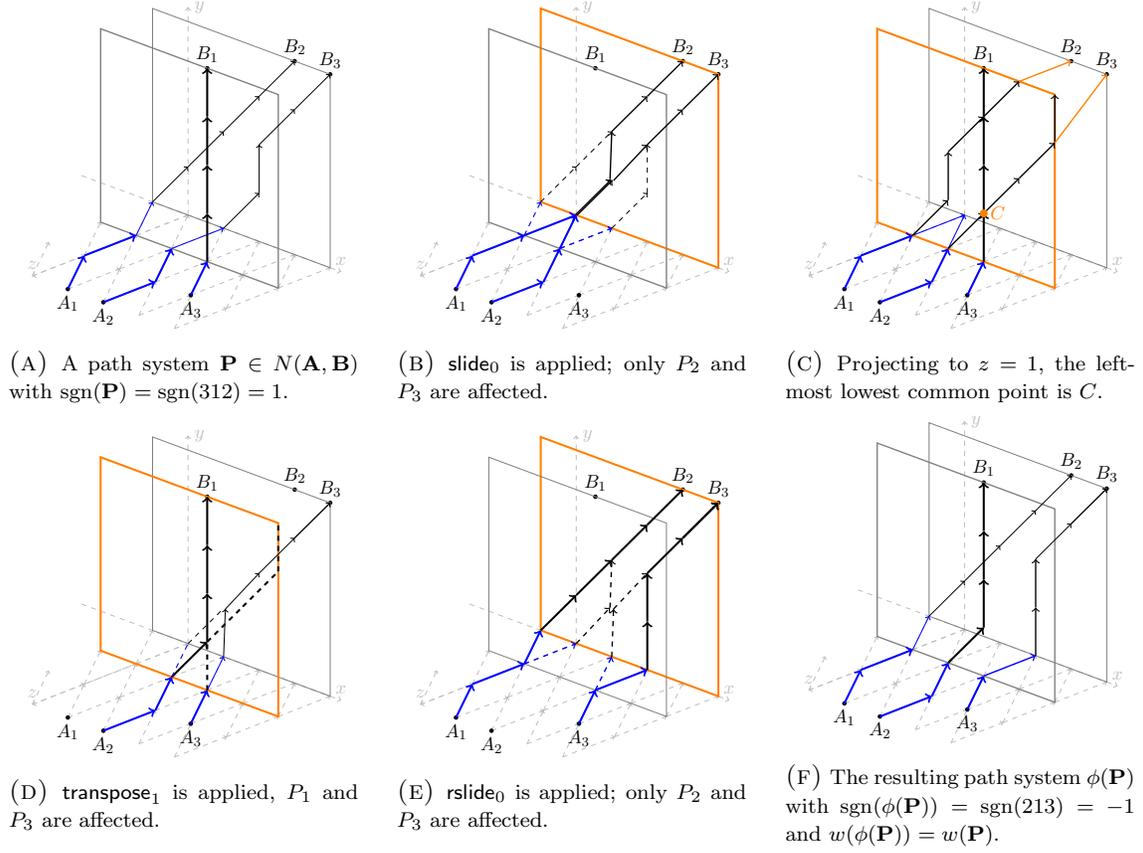

    \begin{center}
        \centering
             \begin{minipage}{0.3\textwidth}
                \centering
                \resizebox{\textwidth}{!}{
                    \pictureInvolutionInit
                }
                \subcaption{\scriptsize A path system $\mathbf{P} \in N(\mathbf{A}, \mathbf{B})$ with $\sgn(\mathbf{P}) = \sgn(312) = 1$. } 
            \end{minipage}
            \quad 
            \begin{minipage}{0.3\textwidth}
                \centering
                \resizebox{\textwidth}{!}{
                    \pictureInvolutionSlide
                }
                \subcaption{\scriptsize $\mathsf{slide}_0$ is applied; only $P_2$ and $P_{3}$ are affected. }
            \end{minipage}
            \quad
            \begin{minipage}{0.3\textwidth}
                \centering
                \resizebox{\textwidth}{!}{
                    \pictureInvolutionProject
            }
            \subcaption{\scriptsize Projecting to $z=1$, the leftmost lowest common point is $C$.}
            \end{minipage}
            \quad
            \begin{minipage}{0.3\textwidth}
                \centering
                \resizebox{\textwidth}{!}{
                    \pictureInvolutionTranspose
            }
            \subcaption{\scriptsize $\mathsf{transpose}_{1}$ is applied, $P_{1}$ and $P_{3}$ are affected. }
            \end{minipage}
            \quad
            \begin{minipage}{0.3\textwidth}
                \centering
                \resizebox{\textwidth}{!}{
                    \pictureInvolutionRevslide
            }
            \subcaption{\scriptsize $\mathsf{rslide}_{0}$ is applied; only $P_{2}$ and $P_{3}$ are affected. }
            \end{minipage}
            \quad
            \begin{minipage}{0.3\textwidth}
                \centering
                \resizebox{\textwidth}{!}{
                    \pictureInvolutionResult
            }
            \subcaption{\scriptsize The resulting path system $\phi(\mathbf{P})$ with $\sgn(\phi(\mathbf{P})) = \sgn(213) = -1$ and $w(\phi(\mathbf{P})) = w(\mathbf{P})$. }
            \end{minipage}
        \caption{An example showing how the involution $\phi$ applies. Here we have $\phi(\mathbf{P}) = s^{-1}_0 \cdot \mathsf{transpose}_1 \cdot s_0 (\mathbf{P}) = \mathsf{rslide}_0 \cdot \mathsf{transpose}_1 \cdot \mathsf{slide}_0 (\mathbf{P})$.}
        \label{fig:torpp}
    \end{center}
\end{figure}

\subsection{Final step: the correspondence between good path systems and skew RPP}

Let us now take any good path system $\mathbf{P}$. We are going to describe how to obtain a unique RPP with the same $(\mathbf{t},\mathbf{x})$ weight. 

Denote by $\mathbf{P'}$ the path system obtained from $\mathbf{P}$ by (slightly) changing the sources to the points $F_i = (i, 0, \lambda'_i)$ and adding single steps $F_i \to A_i = (i-1,0,\lambda'_i - 1)$. Let us consequently perform the operations $\mathsf{slide}_0, \ldots, \mathsf{slide}_{\ell(\lambda) - 1}$ on it.
On each operation $\mathsf{slide}_{k}$ during this process, let us keep track of the $k$-th cut edges belonging to 
the planes $x = i$ for $i = \mu_{k+1} + 1, \ldots, \lambda_{k+1}$, when we project paths from $z = k$ to $z = k + 1$. 
Note that after the operation $\mathsf{slide}_{k}$, $k$-th cut edges do not change in further slide operations. 
Then for $i$ fixed, 
$k$-th cut edges (for $k \in [\mu'_i, \lambda'_i - 1]$) belong to the path $R_i$ on the plane $x = i$  from the point $B_i$ down to $F_i$, formed by following cut edges and (vertical) edges $T_{i,k} T'_{i,k}$ (where type~1 cut edges have $t$-weights, type~2 cut edges have $x$-weights, and vertical edges have weights $1$, see Fig.~\ref{fig:projection}~(b)).

Given the path $R_i$, let us record the $i$-th column (for $i= 1, \ldots, n$) of a skew RPP of shape $\lambda/\mu$ as follows: from top to bottom, we add an entry $s$ if we see an edge of weight $x_s$ (type~2 cute edge) and add an empty entry if we see an edge of $t$ weight (type~1 cut edge). The resulting filling might have some empty entries. To form a proper RPP we then fill each empty entry by the first nonempty entry below it in the same column. The resulting tableau is an RPP with the weight $\prod_i x_i^{c_i} \prod_j t_j^{d_j}$ which is also the weight of $\mathbf{P}$. Since each operation $\mathsf{slide}_k$ preserves the non-increasing order of cut edges at this step, we have $R_i$ is not below $R_{i+1}$ (by $y$-coordinate), which explains why it is a proper RPP.
For an example of this procedure, see Fig.~\ref{fig:torpp}.

Conversely, given any RPP of shape $\lambda/\mu$ we are now describing an inverse procedure how to reconstruct $\mathbf{P}$.
From given RPP it is easy restore the paths $R_i: B_i \to F_i$ by reading its $i$-th column. Note that in our lattice there is a unique path $P'_i$ of weight $w(R_i)$ from $F_i$ to $B_i$. 
Since $R_i$ is never lower than $R_{i+1}$ for all $i$, $k$-th cut edges will not increase (from left to right) for each $k$. Combining this with the fact that $P'_i$ do not intersect on the initial plane $z = \ell(\lambda)$ ($n$ vertical lines), one deduces with Lemma~\ref{slide} that consequent application of the inverse operations $\mathsf{rslide}_{\ell(\lambda)-1}$,\ldots, $\mathsf{rslide}_{0}$ to $\mathbf{P'}$ 
will result in a good path system $\mathbf{P}$ of the same weight.
This completes the proof.
\begin{figure}
        \begin{center}
            \centering
            \begin{minipage}{0.4\textwidth}
                \centering
                \resizebox{\textwidth}{!}{
                    \pictureslideI
                }
                \subcaption{{\scriptsize Initial good path system $\mathbf{P}$.  }}
            \end{minipage}
            \qquad 
            \begin{minipage}{0.4\textwidth}
                \centering
                \resizebox{\textwidth}{!}{
                    \pictureslideA
                }
                \subcaption{\scriptsize Before $\mathsf{slide}_0$ only $P_4$ is affected. }
            \end{minipage}
            \begin{minipage}{0.4\textwidth}
                \centering
                \resizebox{\textwidth}{!}{
                    \pictureslideB
            }
            \subcaption{\scriptsize $\mathsf{slide}_0$ applied, project on  $z = 1$. Cut edges (orange) of each slide are shown on $x = i$. 
            }
            \end{minipage}
            \qquad
            \begin{minipage}{0.4\textwidth}
                \centering
                \resizebox{\textwidth}{!}{
                    \pictureslideC
                }
                \subcaption{\scriptsize $\mathsf{slide}_1$ applied, project on  $z = 2$. 
                The gray path results from $\mathsf{step}_2(P_2,P_3)$.}
            \end{minipage}
            \begin{minipage}{0.4\textwidth}
                \centering
                \resizebox{\textwidth}{!}{
                    \pictureslideD
                }
                \subcaption{\scriptsize $\mathsf{slide}_2$ applied,  project on $z = 3$. \\ ~ }
            \end{minipage}
            \qquad
            \begin{minipage}{0.4\textwidth}
                \centering
                \resizebox{\textwidth}{!}{
                    \pictureslideE
                }
                \subcaption{\scriptsize $\mathsf{slide}_3$ and $\mathsf{slide}_4$ applied. Each orange path (on some $x = i$) converts to a column of RPP.}
            \end{minipage}
            \begin{minipage}{0.3\textwidth}
                \centering
                {\scriptsize
                \begin{align*}
                    \ytableaushort{
                        \none\none\none2,
                        \none11{},
                        {}{}34,
                        134,
                        2
                    } * {3 + 1, 1 + 3, 4, 3, 1}
                \end{align*}
                }
                \subcaption{\scriptsize A reduced RPP: empty boxes of $t$ weights.}
            \end{minipage}
            \begin{minipage}{0.3\textwidth}
                \centering
                {\scriptsize
                \begin{align*}
                    \ytableaushort{
                        \none\none\none2,
                        \none114,
                        1334,
                        134,
                        2
                    } * {3 + 1, 1 + 3, 4, 3, 1}
                \end{align*}
                }
                \subcaption{\scriptsize The resulting RPP. \\ ~}
            \end{minipage}
        \caption{From a good path system to RPP. }
        \label{fig:torpp}
        \end{center}
    \end{figure}
    
{
\begin{remark}
As 
lattice path systems can be converted to RPP (cf.~Fig~\ref{fig:torpp}(F)--(H)), it can be seen that our slide operations on paths are similar to jeu de taquin type slides on tableaux. 
\end{remark}
\begin{remark}
Our proof is new even in the straight case shape $\mu = \varnothing$. The proof in \cite{dy} relied on planar construction combined with the Schur expansion of $g_{\lambda}$ from \cite{lp} (based on RSK). Here, in this special case, the sinks are all on the plane $z = 0$, all nonintersecting path systems are good (so there is no need to apply the involution), and we then apply sequential slide transformations to get RPP's.
\end{remark}
\begin{remark}
It would be interesting to see what other applications can be obtained from the operations and  constructions which we defined. It would also be interesting to see other 3d lattice path systems with similar positivity properties. 
\end{remark}
}

\section{The dual formula}
Set $h_0 = 1$ and $h_i = 0$ for $i < 0$. Let $\Lambda$ be the ring of symmetric functions. Define the automorphism $\varphi : \Lambda \to \Lambda$ via the generators of complete homogeneous symmetric functions $\{h_n \}$ as follows:
$$
\varphi(h_n) = h_n + h_{n-1} + \cdots + h_0, \quad n \ge 1.
$$
It is easy to check that we have 
$$
\varphi^\ell h_n = \sum_{k = 0}^{n} \binom{\ell + k - 1}{k} h_{n - k}, \quad \ell, n \in \mathbb{Z}.
$$
Note also that $\varphi^{-1} h_n = h_n - h_{n-1}$ and for $\ell \ge 0$ we have $\varphi^{\ell} h_n(\mathbf{x}) = h_n(1^\ell, \mathbf{x})$.

\begin{theorem}\label{dual}
Let $n \ge \ell(\lambda), \ell(\mu)$. The following determinantal identity holds
\begin{align}\label{tree}
g_{\lambda/\mu} = \det\left[\varphi^{i - j} h_{\lambda_i - i -\mu_j + j} \right]_{1 \le i, j \le n}. 
\end{align}
\end{theorem}
\begin{proof}
It is known that  that $\{g_{\lambda} \}$ is a basis of $\Lambda$ 
and there is an involutive automorphism (see \cite[Cor.~6.7]{buch}, \cite[Prop.~7.3]{dy2}) $\tau : \Lambda \to \Lambda$ such that for all $\lambda$, $\mu$ 
$$
\tau(g_{\lambda/\mu}) = g_{\lambda'/\mu'}
$$
This involution can be defined via the generators of $\Lambda$ as follows
$$
\tau : g_{(k)} = h_k \longmapsto g_{(1^k)} = e_k(1^{k-1}, \mathbf{x}) = \sum_{\ell} \binom{k-1}{k - \ell}\, e_{\ell}.
$$
Let us now expand the entries of the determinant \eqref{tree}
 $$
 \varphi^{i - j} h_{\lambda_i - i -\mu_j + j}  = \sum_{k } \binom{\lambda_i - \mu_j - k - 1}{\lambda_i - i - \mu_j + j - k} h_k.
 $$
After applying the involution $\tau$  we get the following expression
\begin{align*}
\tau(\varphi^{i - j} h_{\lambda_i - i -\mu_j + j}) &= \sum_{k} \binom{\lambda_i - \mu_j - k - 1}{\lambda_i - i - \mu_j + j - k} \sum_{\ell } \binom{k - 1}{k - \ell} e_{\ell}\\
&= \sum_{\ell} e_{\ell} \sum_{k} \binom{\lambda_i - \mu_j - k - 1}{\lambda_i - i - \mu_j + j - k} \binom{k - 1}{k - \ell}.
\end{align*}
By comparing the coefficients at $[t^{\lambda_i - i - \mu_j + j - \ell}]$ from both sides of the identity $(1 - t)^{j - i -  \ell} = (1 - t)^{j - i} (1 - t)^{-\ell}$ it is easy to see that the following identity holds
$$
\sum_{k} \binom{\lambda_i - \mu_j - k - 1}{\lambda_i - i - \mu_j + j - k} \binom{k - 1}{k - \ell} = \binom{\lambda_i - 1- \mu_j}{\lambda_i - i - \mu_j + j - \ell}.
$$
Hence 
$$
\tau(\varphi^{i - j} h_{\lambda_i - i -\mu_j + j}) = \sum_{\ell}  \binom{\lambda_i - 1- \mu_j}{\lambda_i - i - \mu_j + j - \ell}\, e_{\ell} = e_{\lambda_i - i - \mu_j + j}(1^{\lambda_i - 1- \mu_j}, \mathbf{x}).
$$
and therefore using Theorem~\ref{main} we have
$$
\tau \det\left[\varphi^{i - j} h_{\lambda_i - i -\mu_j + j} \right]_{1 \le i, j \le n}  = \det\left[\tau(\varphi^{i - j} h_{\lambda_i - i -\mu_j + j}) \right]_{1 \le i, j \le n} = g_{\lambda'/\mu'}
$$
which establishes the desired identity. 
\end{proof}

\begin{remark}
Note that entries of the matrix in this dual formula may contain linear combinations of $\{h_n \}$ with negative terms. It would be interesting to find a positive dual formula for $g_{\lambda/\mu}$. It would also be interesting to find a combinatorial proof of this dual formula.
\end{remark}

\subsection{Straight shape case}
We obviously have the following special case for $\mu = \varnothing$.
\begin{corollary}\label{cor:h1} We have
\begin{align}\label{eqh1}
g_{\lambda} = \det\left[\varphi^{i - j} h_{\lambda_i - i + j} \right]_{1 \le i, j \le \ell(\lambda)}. 
\end{align}
\end{corollary}
There is also the following {\it positive} formula that holds for $g_{\lambda}$ (which is not that obvious from the above identity).
\begin{corollary}
 We have
\begin{align}\label{gh}
g_{\lambda}(\mathbf{x})= \det\left[\varphi^{i-1} h_{\lambda_i - i + j} \right]_{1 \le i, j \le \ell(\lambda)}  = \det\left[h_{\lambda_i - i + j}(1^{i-1}, \mathbf{x}) \right]_{1 \le i, j \le \ell(\lambda)}. 
\end{align}
\end{corollary}
This formula follows from \eqref{eqh1} by elementary (column) transformations; we omit these details. Alternatively, it can be proved via a lattice-path construction similar to the one used in \cite{dy} (which is planar and simpler than our main construction here). This formula will be used for deriving one more determinantal identity given in the next section.

\section{Bialternant formula}
\begin{theorem}\label{bialt}
Let $n \ge \ell(\lambda)$. The following formula holds
\begin{align*}
g_{\lambda}(x_1, \ldots, x_n) = \frac{\det\left[h_{\lambda_i + n - i}(1^{i-1}, x_j^{})  \right]_{1 \le i,j \le n}}{\prod_{i < j} (x_i - x_j)}. 
\end{align*}
\end{theorem}
\begin{proof}
The proof follows along the same lines as the classical derivation that relates Jacobi-Trudi and bialternant formulas for Schur polynomials, see \cite[(3.6)]{macdonald}. 

Let $\mathbf{x} = (x_1,\ldots, x_n),$ $\alpha = (\alpha_1, \ldots, \alpha_n) \in \mathbb{N}^n$ 
and consider the matrices 
$$
H_{\alpha} := \left[h_{\alpha_i - n + k}(1^{i-1}, \mathbf{x}) \right]_{1 \le i,k \le n}, \qquad E := \left[(-1)^{n-k} e^{(j)}_{n - k} \right]_{1 \le k,j \le n},
$$
whose entries are given by the following generating series
$$
H_i(t) := \sum_{k = 0}^{\infty} h_{k}(1^{i-1}, \mathbf{x})\, t^k = (1-t)^{1-i} \prod_{k} (1 - x_k t)^{-1}, \quad
E_j(t) := \sum_{\ell=0}^{n-1} e^{(j)}_{\ell} t^{\ell} := \prod_{i \ne j} (1 + x_i t).
$$
Note that we have 
$$
H_i(t) \cdot E_j(-t) = (1-t)^{1-i} (1-x_j t)^{-1} = \sum_{m} h_m(1^{i-1}, x_j) t^m 
$$
and by comparing the coefficients at $[t^{\alpha_i}]$ from both sides of this equality we obtain that 
$$
\sum_{k = 1}^n h_{\alpha_i - n + k}(1^{i-1}, \mathbf{x}) (-1)^{n - k} e^{(j)}_{n - k} = h_{\alpha_i}(1^{i-1}, x_j). 
$$
Therefore,
$$
H_{\alpha} \cdot E = [h_{\alpha_i}(1^{i-1}, x_j)]_{1 \le i, j \le n}, \qquad \det[H_{\alpha}] \cdot \det[E] = \det[h_{\alpha_i}(1^{i-1}, x_j)]_{1 \le i, j \le n}.
$$
Note that $\det[E] = \prod_{i < j} (x_i - x_j)$ and the sequence $\alpha = (\alpha_1,\ldots, \alpha_n)$ given by $\alpha_i = \lambda_i + n - i$ satisfies $\det[H_{\alpha}] = g_{\lambda}$ by the identity \eqref{gh}, and hence we get the desired formula. 
\end{proof}


\begin{remark}[On ribbon formulas]\label{ribbon}
Schur function determinants have a beautiful unifying theory via ribbon decompositions known as the Hamel--Goulden formula \cite{hg}, see also \cite{chen}. It unifies Jacobi--Trudi types, Giambelli hook (see e.g. \cite{macdonald}), and Lascoux--Pragacz \cite{lasp} ribbon formulas. Is there analogous formula for dual Grothendieck polynomials that would generate such identities? Note that there is a nontrivial Giambelli-type identity for $g_{\lambda}$ obtained in \cite{ln}. 
\end{remark}

{
\section*{Acknowledgements}
We are grateful to Askar Dzhumadil'daev, Darij Grinberg, Alejandro Morales, Igor Pak, and Pavlo Pylyavskyy for helpful conversations. We especially thank Darij Grinberg for many useful comments on the paper. We also thank the referee for careful reading of the text and many useful comments. 
}


\end{document}